\documentclass[11pt]{amsart}
\usepackage{amscd}
\usepackage{amsfonts}
\usepackage{latexsym,amssymb, amsmath, amsthm,xcolor,tikz-cd,hyperref,mathtools, geometry}
\hypersetup{
    colorlinks=true,
	allcolors=black,
	linkbordercolor={1 1 1}
}

\begin{document}
\newcommand\nvisom{\rotatebox[origin=cc] {-90}{$ \cong $}}
\newcommand{\Z}{\mathbb{Z}}
\newcommand{\Q}{\mathbb{Q}}
\newcommand{\F}{\mathbb{F}}
\newcommand{\kbar}{\overline{k}}
\newcommand{\pro}{\mathrm{Pro}}
\newcommand{\colim}{\mathrm{colim}}
\newcommand{\Fp}{\mathbb{F}_p}
\newcommand{\Hom}{\mathrm{Hom}}
\newcommand{\Spec}{\mathrm{Spec}}
\newcommand{\et}{\acute{e}t}
\newcommand{\Et}{\acute{E}t}
\newcommand{\xto}{\xrightarrow}
\newcommand{\Gal}{\mathrm{Gal}}

\newtheorem{thm}{Theorem}[section]
\newtheorem{cor}[thm]{Corollary}
\newtheorem{lemma}[thm]{Lemma}
\newtheorem{propn}[thm]{Proposition}
\newtheorem{con}[thm]{Conjecture}
\newtheorem{prop}[thm]{Property}
\newtheorem{sublemma}[thm]{Sub-Lemma}
\theoremstyle{definition}
\newtheorem{example}[thm]{Example}
\newtheorem{nexample}[thm]{Non-Example}
\newtheorem{defn}[thm]{Definition}
\newtheorem{crit}[thm]{Criterion}
\newtheorem{rem}[thm]{Remark}
\newtheorem{qst}[thm]{Question}
\title[Valuative section conjecture, étale homotopy \& Berkovich spaces]{The valuative section conjecture, étale homotopy, and Berkovich spaces}
\author{Jesse Pajwani}
\maketitle
\begin{abstract}We reinterpret a result of Pop and Stix on the $p$-adic section conjecture in terms of Berkovich spaces and fixed points. In doing this, we see a version of the result extends to larger classes of fields, which in turn allows us to prove a valuative section conjecture type result for a larger class of varieties. This adds to the programme to reinterpret anabelian geometry results in terms of étale homotopy types.
\end{abstract}
\setcounter{tocdepth}{1}
\tableofcontents

\section{Introduction}
\subsection{Overview}
Let $k$ be a field with separable closure $\kbar$, and let $X$ be a geometrically connected variety over $k$ with geometric point $\overline{x} \in X(\kbar)$. Let $X_{\kbar}$ denote $X \times_{\Spec(k)} \Spec(\kbar)$ and let $\pi^{\et}_1(X, \overline{x})$ denote the étale fundamental group of $X$. Theorem 6.1 of Section IX of \cite{SGA1} shows that there is a short exact sequence of étale fundamental groups
$$
1 \to \pi^{\et}_1(X_{\kbar}, \overline{x}) \to \pi^{\et}_1(X, \overline{x}) \to \mathrm{Gal}_k = \pi_1^{\et}(\Spec(k),\overline{k}) \to 1.
$$
Suppose $x \in X(k)$ is a $k$-point. Functoriality of $\pi^{\et}_1$ means that $x$ determines a morphism $\mathrm{Gal}_k \to \pi^{\et}_1(X, x)$, and choosing an étale path from $x$ to $\overline{x}$ gives us an isomorphism $\pi_1(X, x) \cong \pi_1(X, \overline{x})$. Choosing a different étale path only changes this isomorphism up to conjugacy, so $x \in X(k)$ gives us a canonical conjugacy class of sections of the above short exact sequence. That is, if we let $\mathcal{S}_{X/k}$ be the set of conjugacy classes of sections of the above short exact sequence, there is a function
$$
h_{X/k}\colon X(k) \to \mathcal{S}_{X/k}.
$$
This map is the basis for the section conjecture, a conjecture first made by Grothendieck in a letter to Faltings, \cite{lettertofaltings}. 
\begin{con}[The section conjecture]
Let $k$ be a number field or a $p$-adic field, and let $X/k$ be a smooth proper geometrically connected curve of genus $\geq 2$. Then the map $h_{X/k}$ is a bijection.
\end{con}
Injectivity of this map clearly depends on the field $k$. For $k$ a number field or $p$-adic field, injectivity is well known and was known at the time of the conjecture, see Equation $7$ of \cite{lettertofaltings} for the number field case, and \S7 of \cite{StixBook} for a detailed treatment. It is thought that Grothendieck's reason for this conjecture is topological, and in part due to the fact that if $X/k$ is a smooth, proper, geometrically connected curve of genus $\geq 2$ over $\mathbb{C}$, then $X(\mathbb{C})$ is a $K( \pi_1(X(\mathbb{C})), 1)$ space. When $k=\mathbb{R}$, Mochizuki proved a version of the above in \cite{MochRSC}, and an alternative topological proof was given in \cite{PalRSC}.

A major result towards the $p$-adic section conjecture is the following theorem of Pop and Stix, the main theorem from \cite{PS}, which links the section conjecture to valuations associated to our variety.
\begin{thm}\label{PopStix}
Let $k$ be a $p$-adic field, and let $X/k$ be a smooth, proper, geometrically connected curve of genus $\geq 2$ with function field $K$. Let $\tilde{X}$ be the universal pro-étale cover of $X$, with function field $\tilde{K}$. Let $s$ be a section of the short exact sequence described above. Then there exists $w$, a valuation on $K$, and $\tilde{w}$, a prolongation of $w$ to $\tilde{K}$, such that $s(\mathrm{Gal}_k) \subseteq D_{\tilde{w}|w}$.  Moreover, $w$ satisfies certain uniqueness properties. 
\end{thm}
The original paper \cite{PS} notes that we can reinterpret the first part of the above theorem as ``any section of the above short exact sequence comes from a $k$-Berkovich point". In this paper we give a proof of this statement which is true over a larger class of fields, Corollary $\ref{maincor}$. The disadvantage of our alternative proof is that we lose the second half of the theorem: from a section, we can obtain a $k$-Berkovich point, but we cannot prove anything about uniqueness. This should not be surprising however, as injectivity of the map $h_{X/k}$ relies on the underlying field, and this is a prerequesite for uniqueness.

The approach we present is different to that in \cite{PS}, and follows the approach of papers such as \cite{StixSchmidt}, \cite{QuickSection}, \cite{HS} and \cite{AmbrusEt} in which anabelian geometry, and in particular the section conjecture, is framed in terms of étale homotopy types. Let $X$ be a variety over a field $k$. Consider the relative étale homotopy type $\Et_{/k}(X)$ as in \cite{HS}. This an object in the pro-homotopy category of simplicial $\Gal_k$ sets that can be thought of as a higher homotopical version of $\pi_1^{\et}$. Define $X(hk)$ to be the set of connected components of the homotopy fixed point space $\Et_{/k}(X)^{h\Gal_k}$. As shown in \cite{HS} and \cite{AmbrusEt}, the map $h_{X/k}$ factors through $X(hk)$ and $X(hk) \to \mathcal{S}_{X/k}$ is a bijection when $X$ is a smooth proper curve of genus $\geq 1$. This allows us to think of $X(k) \to X(hk)$ as a higher homotopical analogue of the map $X(k) \to \mathcal{S}_{X/k}$: see Definition $\ref{hfpdef}$ for more details. 

A goal of this paper is to apply the framework of étale homotopy types to Theorem $\ref{PopStix}$. Let $k$ be a field which is complete with respect to a non-trivial non-archimedean valuation. Let $\kbar$ be its separable closure and let $\hat{\kbar}$ be the completion of the separable closure with respect to the valuation. Let $X$ be a variety over $k$. Consider the Berkovich analytification of a variety, $X^{an}_{\hat{\kbar}}$. This is a topological space with an action of $\Gal_k$ and its points correspond to valuations associated to $X$ (see Lemma $\ref{analytificationvaluation}$). Write $S_\bullet(X^{an}_{\hat{\kbar}})$ to mean the singular simplicial set of $X^{an}_{\hat{\kbar}}$.
\begin{thm}\label{mainthm1}
There is a canonical morphism in the pro-homotopy category of simplicial $\Gal_k$-sets:
$$
\Et_{/k}^{\natural}(X) \to \widehat{S_\bullet(X^{an}_{\hat{\kbar}})}^{\natural},
$$
where $\widehat{-}$ denotes the Artin--Mazur profinite completion functor, and $(-)^{\natural}$ denotes the Postnikov tower functor.
\end{thm}

This gives rise to a map from the homotopy fixed point set $X(hk)$ to the homotopy fixed points of $\widehat{(X^{an}_{\hat{\kbar}})}^{\natural}$. We can then use properties of the underlying topological space of Berkovich spaces to deduce the following theorem.
\begin{thm}\label{mainthm2}
Let $X$ be a smooth proper curve of genus $\geq 1$, and let $x \in X(hk)$ with associated section $s$. Then there exists
$$
\overline{w} = (\overline{w_Y}) \in \varprojlim_Y (Y^{an}_{\hat{\kbar}})^{\Gal_k},
$$
where $Y \to X$ runs over all the covers in the decomposition tower of $s$.
\end{thm}
As a corollary to this, the relationship between valuations and Berkovich analytifications gives us a version of the existence part of Theorem $\ref{PopStix}$.
\begin{cor}
Let $X/k$ be a smooth proper curve of genus $\geq 1$ with function field $K$. Let $\tilde{K}$ be the function field of the universal pro-étale cover of $X$, so that $\pi_1^{\et}(X) \cong \Gal(\tilde{K}/K)$. Let $s\colon \Gal_k \to \pi_1^{\et}(X) \cong \Gal(\tilde{K}/K)$ be a section of the fundamental short exact sequence. There there exists $\tilde{w}$, a valuation on $\tilde{K}$, such that $\tilde{w}$ is a fixed point of the $s(\Gal_k)$ action on $\mathrm{Val}_v(\tilde{K})$. 
\end{cor}
Theorem $\ref{mainthm2}$ drops the requirement for our base field to be $p$-adic. The flexibility of the base field in the above allows us to prove the following extension of Theorem $\ref{mainthm2}$ to a larger class of varieties.
\begin{thm}\label{mainthm3}
Let $X/k$ be a variety such that $X$ factors as
$$
X = C_n \to C_{n-1} \to \ldots \to C_1 \to C_0 = \Spec(k)
$$
where each $C_i \to C_{i-1}$ is a morphism such that for any $L/k$ a field extension, and for any $y \in C_{i-1}(L)$, the base change $(C_i)_y \to \Spec(L)$ is a smooth proper curve of genus $\geq 1$. Let $x \in X(hk)$ and let $s$ be the corresponding section of the short exact sequence. Then $x$ gives rise to a compatible series of points
$$\overline{w} = (\overline{w}_Y) \in \varprojlim_Y (Y^{an}_{\hat{\kbar}})^{\Gal_k},
$$ 
where $Y \to X$ runs over all finite étale covers in the decomposition tower of $s$.
\end{thm}
We also see that these fixed points are compatible with the original homotopy fixed point $x$ in the sense of Definition $\ref{compat}$.
\subsection{Outline}
We first review some properties of étale homotopy types and their relationship to the section conjecture in Section $\ref{ethomotopysection}$. Section $\ref{algtop}$ of this paper is dedicated to showing results about homotopy fixed points for certain simplicial sets. In Section $\ref{berkovichsection}$, we recall results about the underlying topological spaces of Berkovich spaces, and their interaction with the underlying variety. In Section $\ref{comparisonsection}$, we prove Theorem $\ref{mainthm1}$ using a Riemann Existence Theorem type argument. Finally in Section $\ref{fixedpointsection}$ we apply Theorem $\ref{mainthm1}$ to the case of curves. This allows us to use the Section Conjecture for Graphs, \cite{HarpGraph}, to prove Theorem $\ref{mainthm2}$. We then prove compatibility results about the obtained valuation and associated other points, which allows us to prove Theorem $\ref{mainthm3}$. 
\subsection{Notation}
Fix the following notation for the whole paper.  For a category $\mathcal{C}$, write $\mathrm{Pro}(\mathcal{C})$ for the corresponding pro-category as in \S2 of the appendix of \cite{AM}. If $I$ is a cofiltered category and $F\colon I \to \mathcal{C}$ is a functor, write $\{F(i)\}_{i \in I}$ to mean the object $F$ in $\mathrm{Pro}(\mathcal{C})$. When the category $I$ is clear, write $\{F(i)\}_i$. 

Let $s\mathbf{Set}$ denote the category of simplicial sets, and $\mathrm{Ho}(s\mathbf{Set})$ its homotopy category. For $U$ an object of $\mathrm{Pro}(\mathrm{Ho}(s\mathbf{Set}))$, let $\widehat{U}$ denote the profinite completion of $U$. Let $U^\natural$ denote the Postnikov tower of $U$. Let $S_\bullet$ denote the singular functor from the category of topological spaces to the category of simplicial sets. For $\mathcal{C}$ a locally connected site, let $\Pi$ denote the Verdier functor from \S9 of \cite{AM}, so that if $\pi$ is the connected component functor of $\mathcal{C}$ then $\Pi \mathcal{C} := \{ \pi(U_\bullet)\}_{U \in HC(\mathcal{C})}$, where $HC(\mathcal{C})$ is the homotopy category of hypercoverings in $\mathcal{C}$.  

For $L$ any field, let $\Gal_L$ denote the absolute Galois group of $L$. Let $\Gal_L-\mathbf{Set}$ denote the category of sets with a continuous action of $\Gal_L$, and similarly we can form its homotopy category and the pro-homotopy category. By a variety over a field $L$, we mean a reduced separated scheme of finite type over $\Spec(L)$. 

Fix $k$ to be a field which is complete with respect to a non-trivial non-archimedean valuation, $v$. Fix $\kbar$ a separable closure of $k$ and let $\hat{\kbar}$ denote the completion of $\kbar$ with respect to the valuation obtained by extending the valuation on $k$. For $X$ a variety over $k$ and $L/k$ a field extension, let $X_L$ denote the base change of $X$ to $L$. Let $X_{\et}$ denote the small étale site of $X$. For $X$ a variety and $\overline{x}$ a geometric point of $X$, we will write $\pi_1^{\et}(X,\overline{x})$ for the étale fundamental group of $X$ pointed at $\overline{x}$. The results in this paper do not depend on a choice of $\overline{x}$, and we will often suppress the base point for convenience. We will say $s: \Gal_k \to \pi_1^{\et}(X)$ is a \emph{section} if it is a section of the short exact sequence of étale fundamental groups described earlier.
\subsection{Acknowledgements}
The author would like to thank his supervisor, Ambrus Pál, for introducing him to the topic, his continued support and helpful discussions. The author would also like to thank Kirsten Wickelgren for her support and feedback, as well as Kevin Buzzard for his careful reading of the first draft of this paper. Thank you to Johannes Nicaise, Tomer Schlank and Toby Gee, for reading through earlier outlines of this work, and Lambert A'Campo for his useful discussions. Thank you to Marcin Lara for his helpful comments and pointing out some errors in an earlier version of this work. Thank you to the anonymous referee for their feedback. The author was supported by the Engineering and Physical Sciences Research Council grant number EP/S021590/1 as part of The London School of Geometry and Number Theory. The author would also like to thank Imperial College London and the University of Canterbury, where the author was supported by the Marsden Fund Council, managed by Royal Society Te Ap{\= a}rangit. This work appeared as the second chapter of the author's PhD thesis.
\section{Background on étale homotopy types}\label{ethomotopysection}
In this section, we give a review of some properties of étale homotopy types of varieties and their relationship to the section conjecture. For a more detailed treatment of this section, we refer the reader to \cite{HS} and \cite{QuickSection}.

\begin{defn}
Let $X$ be a smooth proper variety over $k$. Write $\Et(X)$ to mean the \emph{étale homotopy type} of $X$ as defined in \S9 of \cite{AM}, which we recall the definition of below. Let $HC(X)$ denote the homotopy category of étale hypercoverings of $X$, then $\Et(X)$ is the element of $\mathrm{Pro}(\mathrm{Ho}(s\mathbf{Set}))$ given by
$$
\Et(X) := \{\pi_0( U_\bullet)\}_{U_\bullet \in HC(X)}.
$$
For $U_\bullet$ a hypercovering of $X$, write $\pi_{0}((U_\bullet)_{\kbar})$ to mean the element of $s\Gal_k-\mathbf{Set}$ given at level $n$ by $\pi_0((U_n)_{\kbar})$ with its canonical action of $\Gal_k$, see Remark 9.13 of \cite{HS}. This gives rise to the \emph{relative étale homotopy type} of $X$, as defined in Definition 9.16 of \cite{HS}
$$
\Et_{/k}(X) := \{ \mathrm{Ex}^{\infty}(\pi_0( (U_\bullet)_{\kbar})) \}_{U_\bullet \in HC(X)},
$$
where  $\mathrm{Ex}^{\infty}$ denotes the Kan replacement functor. Write $\Et^{\natural}(X)$ and $\Et^{\natural}_{/k}(X)$ to mean the Postnikov tower functor applied to $\Et(X)$ and $\Et_{/k}(X)$ respectively.
\end{defn}
\begin{rem}
The étale homotopy type of a variety can be thought of as a higher homotopy analogue of the étale fundamental group: indeed, $\pi_1(\Et(X)) \cong \pi_1^{\et}(X)$ by Corollary 10.7 of \cite{AM}. Indeed, Theorem 1.1 of \cite{StixSchmidt} shows that the ``Isom form'' of Grothendieck's anabelian conjectures hold for hyperbolic curves if we replace their fundamental groups with the étale homotopy type.
\end{rem}

\begin{defn}
Let $G$ be a profinite group. Write $BG$ to mean the object of $\mathrm{Pro}(s\mathbf{Set})$ given by $\{B(G/\Lambda)\}$, where $\Lambda$ runs over all open normal subgroups of $G$, and $B$ is the classifying space functor as in Definition 9.8 of \cite{HS}. Similarly define $EG$ to be the element of $\mathrm{Pro}(sG-\mathbf{Set})$ given by $\{E(G/\Lambda)\}$, where $E$ is the functor taking a group to the universal cover of the classifying space as in Definition 9.7 of \cite{HS}. As pointed out to the author by Marcin Lara, this is a slight abuse of notation. This construction of $BG$ and $EG$ takes into account the profinite topology on $G$. In particular, $\pi_1(BG) \cong G$ as profinite groups, rather than just abstractly as groups.
\end{defn}

\begin{example}
We have the following isomorphisms:
\begin{align*}
\Et^{\natural}(\Spec(k)) &\cong \{ B(\Gal(L/k))\}_{L/k} = B\Gal_k \in \mathrm{Pro}(\mathrm{Ho}(s\mathbf{Set})) \\
\Et_{/k}^{\natural}(\Spec(k)) &\cong \{E(\Gal(L/k))\}_{L/k} = E\Gal_k \in \mathrm{Pro}(\mathrm{Ho}(s\Gal_k-\mathbf{Set})),
\end{align*}
where $L/k$ runs over the cofiltered category of finite Galois extensions of $k$. The first of these is well known, and is computed in Example 3.2 of \cite{PfHomotopyTheory}, though this appears to be known as folklore from earlier than this. The associated computation for the relative étale homotopy type is on page 305 of \cite{HS}.
\end{example}

\begin{defn}
Let $U \in s\Gal_k-\mathbf{Set}$. As in \cite{HS}, write $U^{h\Gal_k}$ to mean the homotopy fixed points of $U$, i.e., the mapping space of $\Gal_k$ equivariant maps $E\Gal_k \to U$, and write $U(hk) : = \pi_0(U^{h\Gal_k})$.  If $U$ has trivial $\Gal_k$ action, then note that $U^{h\Gal_k}$ is the mapping space from $B\Gal_k$ to $U$.

Let $X$ be a variety over $k$. Define $X(hk) : = [ E\Gal_k, \Et_{/k}^{\natural}(X)]_{\Gal_k}$. That is, $X(hk)$ is the set of homotopy classes of $\Gal_k$-equivariant maps from $E\Gal_k$ to $\Et_{/k}^{\natural}(X)$ in the category $\mathrm{Pro}(\mathrm{Ho}(s\Gal_k-\mathbf{Set}))$.
\end{defn}
\begin{defn}\label{hfpdef}
As in Remark 9.29 of \cite{HS} and Definition 2.4 of \cite{AmbrusEt}, we define a map $h_{X}\colon X(k) \to X(hk)$ as follows. A $k$-point $x \in X(k)$ is a map $x\colon \Spec(k) \to X$, so an element of $\Hom_{/k}(\Spec(k), X)$. A homotopy fixed point of $\Et_{/k}^{\natural}(X)$ is a $\Gal_k$ equivariant map $\Et_{/k}^{\natural}(\Spec(k)) \to \Et_{/k}^{\natural}(X)$ in $\mathrm{Pro}(\mathrm{Ho}(s\Gal_k-\mathbf{Set}))$. Applying the functor $\Et_{/k}^{\natural}(-)$ gives rise to a map $\Hom_{/k}(\Spec(k), X) \to [E\Gal_k, \Et_{/k}^{\natural}(X)]_{\Gal_k}$, i.e., a function of sets $X(k) \to X(hk)$.
\end{defn}
\begin{defn}
Let $X$ be geometrically connected. We can construct a natural map $X(hk) \to \mathcal{S}_{X/k}$ as follows. Let $\Et_{/k}^{1}(X)$ denote the Postnikov piece at level $1$, so that we have $\pi_i(\Et_{/k}^{1}(X))=0$ if $i\neq1$ and $\pi_i(\Et_{/k}^{1}(X))=\pi_i(\Et_{/k}^{\natural}(X))$ if $i=1$. By construction of the Postnikov tower, we have a canonical map $\Et_{/k}^{\natural}(X) \to \Et_{/k}^1(X)$, and therefore a map $X(hk) \to \pi_0( \Et_{/k}^1(X)^{h\Gal_k})$. As noted in Corollary 9.86 of \cite{HS}, Theorem 7.6(ii) of \cite{AmbrusEt} and Theorem 4.3 of \cite{QuickSection}, this second set can be identified with the set of conjugacy classes of sections $\mathcal{S}_{X/k}$. The composition $X(k) \to X(hk) \to \mathcal{S}_{X/k}$ is precisely the section map from the introduction (see Theorem 7.6 of \cite{AmbrusEt}). In this sense, the set $X(hk)$ can be thought of as a homotopy theoretic refinement of $\mathcal{S}_{X/k}$.

Let $x \in X(hk)$, and let $s \colon \Gal_k \to \pi_1^{\et}(X, \overline{x})$ be a section. We say $s$ is a \emph{section coming from $x$} if $s$ lies in the conjugacy class of sections which is the image of $x$ under the map $X(hk) \to \mathcal{S}_{X/k}$. 
\end{defn}

\begin{defn}
We say that a variety $X$ is a \emph{$K(\pi,1)$ variety} if $\pi_n(\Et(X))=0$ for all $n\neq1$. In the language of \cite{HS}, this is saying that $\Et_{/k}^\natural(X) \cong \Et_{/k}^1(X)$. If $X$ is a $K(\pi,1)$ variety then the map $X(hk) \to \mathcal{S}_{X/k}$ is naturally a bijection, since it is the map induced by taking homotopy fixed points of the map $\Et_{/k}^{\natural}(X) \to \Et_{/k}^1(X)$. As noted in Lemma 2.7(a) of \cite{StixSchmidt}, smooth proper curves of genus $\geq 1$ are $K(\pi,1)$ varieties. 
\end{defn}

\section{Algebraic Topology Results}\label{algtop}
In this section, we establish some results about sites of topological spaces and homotopy fixed points which we will use later: Lemmas $\ref{verdierfunctorsingular}$ and $\ref{pfMil}$. Both results are purely about algebraic topology, which we later apply to the underlying topological space of Berkovich analytifications of varieties.

Let $T$ be a compact, locally contractible, topological space. Let $T_t$ denote the small site where the objects are topological spaces $U \xto{p_U} T$ such that $p_U$ is a local isomorphism. Coverings in this site are given by $U \xto{p_U} T$ such that $p_U$ is surjective. Similarly, let $T_{ft}$ be the small site where the objects are topological spaces $U \xto{p_U} T$ such that $p_U$ is a local isomorphism with finite fibres, i.e., $p_U^{-1}(t)$ is finite for all $t \in T$. Coverings in this site again are given by surjective maps.
\begin{lemma}
Let $V \xto{p_V} T$ be a local isomorphism with finite fibres that is a covering of $T$. Then there exists $U \xto{f} V$ such that $U \xto{p_V \circ f} T$ is a covering of $T$ which is a local isomorphism with finite fibres such that every connected component of $U$ is contractible.
\end{lemma}
\begin{proof}
Note that $V$ is locally contractible, since $p_V$ is a local isomorphism. For every $p \in T$, let $T_p$ be a contractible neighbourhood of $p$. Since $p_V$ is a local isomorphism with finite fibres, shrinking $T_p$ if necessary, we may also assume that $p_V^{-1}(T_p)$ is given by a finite disjoint union of copies of $T_p$.

Since the $T_p$s form an open cover of $T$, we can find $T_{p_1}, \ldots, T_{p_n}$ such that $T = \bigcup_{i=1}^n T_{p_i}$. Let $U_i := p_V^{-1}(T_{p_i})$ and let $f_i$ denote the inclusion $f_i\colon U_i \hookrightarrow V$. Each connected component of $U_i$ is contractible by construction. Let $U$ denote the disjoint union of the $U_i$s, and let $f$ be the coproduct of the inclusion maps. The result is then clear.
\end{proof}

\begin{lemma}\label{verdierfunctorsingular}
Let $T$ be a compact locally contractible topological space. Then there is an isomorphism $\Pi T_t \cong \Pi T_{ft} \cong S_\bullet T$ in $\mathrm{Pro}-\mathrm{Ho}(s\mathbf{Set}))$. 
\end{lemma}
\begin{proof}
Theorem 12.1 of \cite{AM} tells us that $\Pi T_t \cong S_\bullet T$. Moreover if $U_\bullet \to T$ is a hypercovering in $T_t$ such that each connected component of $U_q$ is contractible, then $\pi_0(U_\bullet) \cong S_\bullet T$. By repeatedly applying the previous lemma, we see every hypercovering with finite fibres $V_\bullet \to T$ has a refinement $U_\bullet \to V_\bullet$, where $U_\bullet$ is a hypercover with finite fibres such that each connected component of each $U_q$ contractible. For such $U_\bullet$, we have $\pi_0(U_\bullet) \cong S_\bullet T$. The fact that these hypercoverings are cofinal implies $\Pi T_{ft} \cong S_\bullet T$, as required. 
\end{proof}

We will also need the following theorems about homotopy fixed points.
\begin{lemma}\label{decomp}
Let $G$ be a profinite group and let $H$ be an open normal subgroup of $G$. Then for any $G$-simplicial set $X$,
$$
X^{hG} = (X^{hH})^{hG/H}
$$
\end{lemma}
\begin{proof}
This is Proposition 3.3.1, part 4 of \cite{HFP}. It should be noted that the paper refers to $G$-spectra rather than $G$-simplicial sets, however the proof only uses standard properties of model categories which hold in this setting.  
\end{proof}
We also need an extension of Miller's theorem.
\begin{thm}[Miller's Theorem, Theorem A of \cite{Miller}]
Let $G$ be a finite group, and let $X$ be a finite CW-complex. Then the canonical morphism $X \to \mathrm{Map}(BG, X)$ is a weak homotopy equivalence.
\end{thm}
\begin{cor}\label{pfMil}
Let $X$ be a finite CW-complex and let $H$ be a profinite group. Then the canonical map $X \to \mathrm{Map}(BH, X)$ is an isomorphism in $\mathrm{Pro}-\mathrm{Ho}(s\mathbf{Set})$. 
\end{cor}
\begin{proof}
We identify $BH \cong \{ B(H/\Lambda)\}_{\Lambda} \in \mathrm{Pro}-\mathrm{Ho}(s\mathbf{Set})$, so that there is an isomorphism
$$
\mathrm{Map}(BH, X)  \cong \mathrm{colim}_{\Lambda} \mathrm{Map}(B(H/\Lambda), X). 
$$
Each object in the colimit is isomorphic to $X$ by Miller's theorem, and all the transition maps in the diagram are be isomorphisms as well, since the diagram is filtered. Therefore $X \to \mathrm{Map}(BH, X)$ must also be an isomorphism.
\end{proof}
As well as this, we can also weaken the condition on $X$ to be a finite CW-complex.
\begin{lemma}\label{profmil}
Let $X$ be a finite CW-complex and let $H$ be a profinite group. Then there is an isomorphism in $\mathrm{Pro}-\mathrm{Ho}(s\mathbf{Set})$:
$$\mathrm{Map}(BH, \widehat{X}) \cong \widehat{X},$$
where $\widehat{X}$ denotes the profinite completion of $X$ in the sense of Artin--Mazur.
\end{lemma}
\begin{proof}
In the case that $H$ is finite, this is Theorem 3.1 of \cite{FM}. It should be noted that \cite{FM} proves this for the Sullivan profinite completion. However, the proof shows that for any $X_{\alpha}$ in the diagram defining $\widehat{X}$, the map $X_\alpha \to \mathrm{Map}(BH, X_\alpha)$ is a weak equivalence, therefore the proof holds for both the Sullivan or Artin--Mazur profinite completion. For $H$ profinite, we appeal to the same argument as for Corollary $\ref{pfMil}$.
\end{proof}

\section{The topology of Berkovich analytifications}\label{berkovichsection}
Theorems $\ref{mainthm2}$ and $\ref{mainthm3}$ require us to look at Berkovich analytifications of varieties, and in particular their underlying topological spaces. In this section we recall the definition of a Berkovich analytification of a variety, as well as some propreties about their topologies.

\begin{defn}\label{valuationdefinition}
Let $K$ be a field. Following Appendix A of \cite{PS}, a \emph{valuation} on $K$ is a function $w: K \to \Gamma \cup \{\infty\}$, where $\Gamma$ is a totally ordered abelian group, such that:
\begin{enumerate}
\item $w(\alpha)=\infty$ if and only if $\alpha=0$.
\item $w(\alpha \beta) = w(\alpha) + w(\beta)$.
\item $w(\alpha + \beta) \geq \mathrm{min}(w(\alpha), w(\beta))$ with equality if and only if $w(\alpha) \neq w(\beta)$.
\end{enumerate}
The ring of integers $\mathfrak{o}$ of $w$ is the set of elements $\alpha$ with $w(\alpha) \geq 0$, and the \emph{rank} of $w$ is the Krull dimension of $\mathfrak{o}$. We will only consider valuations of finite rank. By Chapter 6 \S4.4 Proposition 8 of \cite{Bourbaki}, we may view $\Gamma$ as a subgroup of $\mathbb{R}^d$ where $d=\mathrm{rank}(w)$ and we give $\mathbb{R}^d$ the lexicographic ordering. 

Fix $v$ to be a valuation on a field $k$, and let $K/k$ be a field extension. Following page 3 of \cite{PS}, define the set $\mathrm{Val}_v(K)$ to be the set of valuations $w$ on $K$ such that $w|_{k^\times}=v$. 

Let $w$ be a valuation of rank $1$ on a field $K$. Since $w$ has rank $1$, we may think of $w$ as a function $w: K \to \mathbb{R} \cup \{\infty\}$. We obtain a multiplicative seminorm on $k$, $|\cdot|_w$ by setting $|\lambda|_w = \mathrm{exp}(-w(\lambda))$ if $\lambda \neq 0$ and $|\lambda|_w = 0$ if $\lambda=0$. Similarly, from a multiplicative seminorm on $K$ we may obtain a valuation of rank $1$ on $K$ setting $w = -\mathrm{log}(|\lambda|_w)$. As a result, we will think of multiplicative seminorms and rank $1$ valuations as interchangable.
\end{defn}
Unless otherwise specified, fix $k$ to be a field which is complete with respect to a non-trivial non-archimedean valuation $v$.
\begin{defn}\label{berkanalytification}
Let $X/k$ be a variety. As on page 194 of \cite{HL}, define the \emph{Berkovich analytification} $X^{an}$ of $X$ to be the topological space whose points are pairs $(x,w)$, where $x$ is a point on the underlying topological space of $X$, and $w$ is a valuation of rank $1$ on the residue field of $x$ extending the valuation on $k$. There is a natural function of sets $\pi\colon X^{an} \to X$, given by $(x,w) \mapsto x$. The topology on $X^{an}$ is defined to be the coarsest topology such that:
\begin{enumerate}
\item The function $\pi: X^{an} \to X$ described above is continuous.
\item For any $U \subseteq X$ an open set, and $f \in \mathcal{O}_X(U)$, the function $\pi^{-1}(U) \to \mathbb{R} \cup \{\infty\}$ given by $(x,w) \mapsto w(f(x))$ is continuous, where $f(x)$ denotes the image of $f$ in the residue field at $x$.
\end{enumerate}
It should be noted that typically Berkovich analytifications are defined in terms of multiplicative seminorms, however as noted in Definition $\ref{valuationdefinition}$, we may instead characterise them using valuations of rank $1$. 

Berkovich analytifications have properties that reflect their underlying variety. Theorem 3.4.8 of \cite{SpectralGeometry} tells us that we have the following equivalences:
\begin{enumerate}
\item The variety $X$ is proper if and only if $X^{an}$ is Hausdorff and compact. 
\item The variety $X$ is connected if and only if $X^{an}$ is path connected.
\item The dimension of the variety $X/k$ is equal to the topological dimension of $X^{an}$.
\end{enumerate}
Berkovich analytifications naturally also have the structure of a locally ringed space, and the original definition in Theorem 3.4.1 of \cite{SpectralGeometry} makes use of this structure. The equivalence of this definition with the definition above follows by Remark 3.4.2 and Corollary 3.4.5 of \cite{SpectralGeometry}. We refer the reader to \cite{SpectralGeometry} for additional details on Berkovich spaces. We will only use the locally ringed structure in order to talk about the étale site on a Berkovich analytification and for related proofs. Note that if $X/k$ is a variety, then \S1.5 of \cite{Berk} implies that $X^{an}$ is a good Berkovich space in the sense of Remark 1.2.16 of \cite{Berk}. In particular, $X^{an}$ admits an open cover by sub-Berkovich spaces isomorphic to $\mathcal{M}(A)$, where $A$ is an affinoid $k$-algebra (see \S2.1 of \cite{SpectralGeometry}) and $\mathcal{M}$ denotes the functor from \S1.2 of \cite{SpectralGeometry}.
\end{defn}

\begin{lemma}\label{analytificationvaluation}
Let $X/k$ be a smooth proper curve with function field $K$. Then we can naturally identify the topological space $X^{an}$ with a subset of $\mathrm{Val}_v(K)$.
\end{lemma}
\begin{proof}
Let $(x,w)$ be a point of $X^{an}$. If $x$ is the generic point of $X$, then $w$ is a discrete valuation of rank $1$ on $K$ extending the valuation on $k$ by definition. 

Suppose then that $x$ is not the generic point of $X^{an}$. Then $x$ is a closed point of $X$, and $w$ is the unique valuation on the residue field $k(x)$ such that $w|_{k^\times}=v$. Since $x$ is a closed point on $x$, it defines a discrete valuation on $K$ given by sending a function $f$ to the order of vanishing of $f$ along $x$. Let $\mathcal{O}_x$ denote the ring of integers of this valuation. Write $\mathrm{ord}_x$ for this valuation and let $\pi$ denote a uniformiser of $\mathrm{ord}_x$, so that $\pi^{-\mathrm{ord}_x(f)}f$ lies in $\mathcal{O}_x^\times$. Write $\overline{\pi^{-\mathrm{ord}_x(f)}f}$ to mean the image of $\pi^{-\mathrm{ord}_x(f)}f$ under the reduction map $\mathcal{O}_x^\times \to k(x)^\times$. We obtain a rank $2$ valuation on $K$, given by
$$
f \mapsto (\mathrm{ord}_x(f), w(\overline{\pi^{-\mathrm{ord}_x(f)}f})) \in \Z \times \Z,
$$
where we equip $\Z \times \Z$ with the lexicographic ordering. Sending the point $(x,w)$ to this valuation gives a map $X^{an} \to \mathrm{Val}_v(K)$ as required. 
\end{proof}

Note that $\mathrm{Gal}_k$ acts on $\hat{\kbar}$ by isometries, so $X_{\hat{\kbar}}$ admits a $\mathrm{Gal}_k$ action. By functoriality of $(-)^{an}$, $X_{\hat{\kbar}}^{an}$ has a $\mathrm{Gal}_k$ action. In this section, we explore the properties of the $\Gal_k$ action on $X_{\hat{\kbar}}^{an}$. If $L/k$ is a finite Galois extension, we have a strong result describing the homotopy type of $X_L^{an}$ as a $\Gal(L/k)$-topological space. 
\begin{thm}[Theorem 11.1.1(2) of \cite{HL}]\label{Skeleton}
Let $X/k$ be a variety, and let $X^{an}_L$ denote the analytification of $X_L$ for $L/k$ a finite Galois extension. Then there exists $S_L \subseteq X_L^{an}$ such that $S_L$ is homeomorphic to a finite $\mathrm{Gal}(L/k)$-CW-complex and the Berkovich analytification $X_L^{an}$ admits a $\mathrm{Gal}(L/k)$-equivariant strong deformation retraction onto $S_L$. 
\end{thm}
\begin{rem}
Since the map $X_L^{an} \to S_L$ is a $\mathrm{Gal}(L/k)$-equivariant strong deformation retraction, it is clear that $S_L^{\mathrm{Gal}(L/k)} \subseteq (X_L^{an})^{\mathrm{Gal}(L/k)}$.  Suppose now that $X$ is a curve. Since the dimension of $X_L$ is equal to the topological dimension of $X_L^{an}$, this forces $X_L^{an}$ to have the homotopy type of a graph. Corollary 2.4.10 of \cite{NicaiseBerkSkel} allows us to explicitly determine the (non-equivariant) homotopy type of $X_L^{an}$ by consider models of our variety, constructing a CW-complex lying in $X_L^{an}$ and a strong deformation retraction of $X_L^{an}$ onto this CW-complex. However, there is no guarantee that the construction in this situation is $\Gal(L/k)$-equivariant.
\end{rem}
We wish to describe the $\Gal_k$ action on $X^{an}_{\hat{\kbar}}$ in terms of the actions of $\Gal(L/k)$ on $X^{an}_L$, where $L/k$ runs over all finite Galois extensions. We first note we have the following.

\begin{lemma}
Let $K/\kbar$ be a finitely generated field extension, and let $v$ denote the valuation on $\kbar$ given by the prolongation of the valuation on $k$. Let $w$ be a valuation of rank $1$ on $K$ such that $w|_k = v$. Then $w$ extends uniquely to a valuation on $K \otimes_{\kbar} \hat{\kbar}$.
\end{lemma}
\begin{proof}
Since $w$ is a valuation of rank $1$ on $K$, we can think of $w$ as a function $K^\times \to \mathbb{R}^+$ (see Definition $\ref{valuationdefinition}$), and therefore $w$ equips $K^\times$ with a topology, which is the coarsest topology on $K^\times$ such that $w$ is continuous. Let $x \in K \otimes_{\kbar} \hat{\kbar}$, and write $x=\sum_{i=1}^m (f_i \otimes \lambda_i)$. Since $\lambda_i \in \hat{\kbar}$, there exists a sequence $(\mu_{i_n})$ with $\mu_{i_n} \in \kbar$ such that $\mu_{i_n}$ converges to $\lambda_i$ as $n$ tends to infinity.

Let $x_n := \sum_{i=1}^m (f_i \otimes \mu_{i_n})$. Since $\mu_{i_n} \in \kbar$, we see that $x_n \in K$, so $w(x_n)$ is well defined. Each sequence $(f_i \otimes \mu_{i_n})$ converges to $f_i \otimes \lambda_i$ as $n$ tends to infinity and so the sequence $(x_n)$ converges to $x$ as $n$ tends to infinity. Since $w$ is a continuous function, the sequence $w(x_n)$ also converges, and therefore letting $w(x)$ be the limit of the sequence $w(x_n)$ gives us the desired extension.
\end{proof}

\begin{defn}
For $X$ a variety over $k$, we will abuse notation and write $X^{an}_{\kbar}$ to mean the topological space whose points are pairs $(\overline{x},w)$, where $\overline{x}$ is a point on the underlying topological space of $X_{\kbar}$, and $w$ is a valuation of rank $1$ on the residue field of $\overline{x}$ extending the valuation on $\kbar$. The topology is given in the same way as for Berkovich analytifications. 
\end{defn}
This is an abuse of notation, and $X^{an}_{\kbar}$ fails to be a Berkovich analytification, as the ground field is not complete with respect to the valuation. When trying to define a structure sheaf on $X^{an}_{\kbar}$, we would run into problems. However, as a topological space, we have the following.
\begin{cor}
There is a $\Gal_k$-equivariant isomorphism of topological spaces
$$
X^{an}_{\kbar} \cong X^{an}_{\hat{\kbar}}.
$$
\end{cor}
\begin{proof}
This follows by the previous lemma and the fact that $\Gal_k$ acts on $\hat{\kbar}$ by isometries. 
\end{proof}
\begin{lemma}\label{berkdesc}
Let $X/k$ be a geometrically connected variety. Then the canonical action of $\Gal_k$ on $X^{an}_{\hat{\kbar}}$ is continuous, and $X^{an}_{\hat{\kbar}}/\mathrm{Gal}_k \cong X^{an}$. 
\end{lemma}
\begin{proof}
The previous corollary allows us to replace $X^{an}_{\hat{\kbar}}$ with $X^{an}_{\kbar}$. Without loss of generality, we may assume that $X$ is affine, $X=\Spec(A)$. Since $X$ is geometrically connected, $A$ embeds into its field of fractions $K$, and we may identify $\mathrm{Gal}_k \cong \Gal(K \otimes_k \kbar / K)$. By definition, $X^{an}_{\kbar}$ has a basis of open sets given by $D(f)$ for $f \in A \otimes_k \kbar$, where
$$
D(f) = \{ (\mathfrak{p}, w): f \not\in \mathfrak{p}, w(f) < 1\}.
$$
Since $f \in A \otimes_k \kbar \subseteq K \otimes_k \kbar$, the extension $K(f)/K$ is finite and separable. Let $L$ denote the Galois closure of $K(f)/K$, and let $\sigma \in \Gal(K \otimes_k \kbar / L)$. Then since $f \in L$,
$$
\sigma (D(f)) = D(\sigma f) = D(f).
$$
In particular, $\Gal(K \otimes_k \kbar / L)$ is contained in $\mathrm{Stab}(D(f))$.

Since $L/K$ is Galois, $\Gal(K \otimes_k \kbar / L)$ is a normal subgroup of $\Gal(K \otimes_k \kbar / K)$, so consider the quotient map $\Gal(K \otimes_k \kbar / K) \to \Gal(L/K)$, where we equip $\Gal(L/K)$ with the discrete topology. This quotient map is continuous by definition of the topology on $\Gal(K \otimes_k \kbar / K)$. Since $\mathrm{Stab}(D(f))$ is the inverse image of $\mathrm{Stab}(D(f))/\Gal(K \otimes_k \kbar / L)$ under this quotient map, this implies that $\mathrm{Stab}(D(f))$ is open in $\Gal(K \otimes_k \kbar) = \Gal_k$, which proves that the $\Gal_k$ action is continuous. 

It remains to show the statement about quotients. In the case that $X$ is affine, the statement about quotients follows from Corollary 1.3.6 of \cite{SpectralGeometry}, and the general case follows since $X_{\hat{\kbar}}$ can be covered by affine $\mathrm{Gal}_k$-equivariant open sets.
\end{proof}

\begin{lemma}
Let $G$ be a profinite group with a continuous action on a compact Hausdorff topological space $X$ such that $X/\Lambda$ is Hausdorff for all $\Lambda$ an open normal subgroup of $G$. Then $X \cong \varprojlim_{\Lambda} X/\Lambda$ as a $G$-topological space, where $\Lambda$ runs over all open normal subgroups of $G$.
\end{lemma}
\begin{proof}
There is trivially a morphism $X \to \varprojlim_{\Lambda} X/\Lambda$. Note that $X/\Lambda$ is Hausdorff for each $\Lambda$, and so $\varprojlim_{\Lambda} X/\Lambda$ is Hausdorff. Moreover, $X$ is compact, therefore we only need to show that this morphism is a bijection.

We first show that it is an injection. Suppose that $x, y \in X$ with $x \neq y$. We want to show that there is some open normal subgroup $\Lambda$ such that $x$ and $y$ are not in the same $\Lambda$ orbit. Let $S := \{g \in G: g \cdot x =y\}$. Note that $S$ is closed: $S = \phi_x^{-1}(\{y\})$, where $\phi_x\colon G \to X$ is the map $g \mapsto g\cdot x$. Then $H = G \setminus S$ is open, and contains $1_G$. Since open normal subgroups form a basis of neighbourhoods of $1_G$, there exists $\Lambda$ such that $\Lambda \subseteq H$. Then $x,y$ are not in the same $\Lambda$ orbit, by construction.

To show it is a surjection, let $(x_\lambda)_{\Lambda} \in \varprojlim_{\Lambda} X/\Lambda$. Let $A_{\Lambda} := p_{\Lambda}^{-1}( x_{\Lambda})$, where $p_{\Lambda}$ is the quotient by $\Lambda$ map $X \to X/\Lambda$. The $A_{\Lambda}$ form a cofiltered system and are each non-empty and compact (since they are closed in a compact space), so $\bigcap_{\Lambda} A_{\Lambda}$ is non empty.
\end{proof}
\begin{cor}\label{proberk}
Let $X/k$ be a proper geometrically connected variety. Then there is an isomorphism of topological spaces with a $\Gal_k$ action $X^{an}_{\hat{\kbar}} = \varprojlim X^{an}_L$, where $L/k$ runs over all finite Galois extensions of $k$.
\end{cor}
\begin{proof}
For any finite Galois extension $L/k$, $X^{an}_L = X^{an}_{\hat{\kbar}} / \mathrm{Gal}_L$ by Lemma $\ref{berkdesc}$. Since $X$ is proper, $X^{an}_L$ is Hausdorff and compact by property $1$ in Definition $\ref{berkanalytification}$. This then follows by the previous theorem.
\end{proof}
\begin{cor}\label{BerkovichProObject}
There is an isomorphism in $\mathrm{Pro}-\mathrm{Ho}(s\mathrm{Gal}_k-\mathbf{Set})$:
$$
S_\bullet|X_{\hat{\kbar}}^{an}| \cong \{S_\bullet |X_L^{an}|\}_{L/k}
$$
where $L/k$ runs over all finite Galois extensions of $k$.
\end{cor}

The final construction of this section allows us to obtain points on varieties associated to points in the Berkovich analytification. The following is a similar trick to that in \cite{AmbrusRealII} to produce field valued points on the underlying variety from points on an associated analytification.
\begin{defn}\label{markedpt}
Let $X/k$ be a variety, and let $\gamma \in X^{an}$. Then $\gamma$ is a pair $(p,w)$, where $p$ is a point in the scheme $X$, and $w$ is a valuation on the residue field $k(p)$ at $p$ extending the valuation on $k$. Let $k(\gamma)$ denote the completion $k(p)_w$. This furnishes a morphism, $\Spec(k(\gamma)) \to X$, the set theoretic image of which is precisely $p$. We refer to this point as $\gamma_X$. Write $X^\gamma$ to mean the fibre product
\begin{center}
\begin{tikzcd}
X^\gamma \ar[r] \ar[d] & X \ar[d] \\
\Spec(k(\gamma)) \ar[r] & \Spec(k).
\end{tikzcd}
\end{center}
By the universal property of the fibre product, the morphism $\gamma_X: \Spec(k(\gamma)) \to X$ gives us a map $\Spec(k(\gamma)) \to X^\gamma$, given by $\gamma_X \times \mathrm{Id}$. In particular, $X^\gamma$ has a canonical $k(\gamma)$ point. When it is clear, we will also refer to this marked point as $\gamma_X$.
\end{defn}\section{An Étale Comparison Morphism}\label{comparisonsection}

This section is dedicated to showing the existence of a morphism from $\Et^{\natural}(X)$ to $\widehat{S_\bullet(X^{an})}^{\natural}$. The non-equivariant version of this morphism is much easier to construct, so we begin with this. We then construct an ``equivariant version", which draws on a version of Proposition 9.19 of \cite{HS} and applies it to the non-equivariant construction.

\subsection{A Non Equivariant Comparison Morphism}
In this section, we show a theorem comparing the homotopy types of $\Et(X)$ with the homotopy type of its Berkovich space. This should be thought of as a version of Theorem 12.9 of \cite{AM} for Berkovich spaces, as well as an analogue of Section 10 of \cite{AmbrusRealI}. 

 Let $X_{\et}$ denote the small étale site over $X$, $X^{an}_{\et}$ the small étale site over $X^{an}$, and $X^{an}_{ft}$ the site from Section $\ref{algtop}$. Lemma 2.6 of \cite{dejong} tells us that if $\mathcal{Y} \to X_L^{an}$ is a finite local isomorphism of topological spaces, then there is a Berkovich space, $\mathcal{Y}^{an}$ with underlying topological space $\mathcal{Y}$, and a map $\mathcal{Y}^{an} \to X_L^{an}$ which is a finite local isomorphism whose underlying map on topological spaces is $\mathcal{Y} \to X_L^{an}$.  This means we obtain a map of sites $X^{an}_{\et} \to X^{an}_{ft}$, given by inclusion of categories $X^{an}_{ft} \hookrightarrow X^{an}_{\et}$. Moreover, by Proposition 3.3.11 of \cite{Berk}, we see that there is a functor $X_{\et} \to X^{an}_{\et}$, given by $U \mapsto U^{an}$, and by Proposition 2.6.8 of \cite{Berk}, this gives us a morphism of sites $X^{an}_{\et} \to X_{\et}$. That is, we have a diagram of sites.
\begin{center}
\begin{tikzcd}
& X^{an}_{\et} \ar[rd, "\phi"] \ar[ld, "\psi"'] & \\
X_{\et} && X^{an}_{ft}.
\end{tikzcd}
\end{center}
\begin{thm}[Corollaries 5.27 and 5.31 of \cite{NAEtHop}]\label{nonequivariantisomorphism}
Let $X$ be a proper variety over $k$. The morphism of sites $X^{an}_{\et} \to X_{\et}$ induces an isomorphism in  $\mathrm{Pro}(\mathrm{Ho}(s\mathbf{Set}))$
$$
\widehat{ \Pi X^{an}_{\et}}^{\natural} \cong \widehat{\Et^{\natural}(X)}.
$$
\end{thm}
\begin{rem}
It should be noted that \cite{NAEtHop} instead works with the shape of a site in the $\infty$-categorical sense and defines ``$\Et(X^{an})$" to be the shape of $X^{an}_{\et}$, and similarly $\Et(X)$ is defined to be the shape of $X_{\et}$. However, Theorem 2.77 of loc. cit. implies that if $S$ is a site then $\Pi S$ naturally corepresents the shape of $S$, and since the co-Yoneda functor is fully faithful, the isomorphism of shapes gives us an isomorphism of the underlying pro-homotopy types.

Note that this theorem requires that the valuation on $k$ is non-trivial valuation. However, Remark $\ref{trivialvaluation}$ shows that we may obtain a trivial version of the main theorem of this paper when the valuation on the base field is trivial.
\end{rem}
\begin{cor}\label{cor1}
For $X/k$ a proper variety, there is a canonical morphism, natural in $X$:
$$
\widehat{\Et^{\natural}(X)} \to \widehat{\Pi X^{an}_{ft}}^{\natural}.
$$
\end{cor}
\begin{proof}
The morphism is given by $\widehat{\Pi \phi} \circ \widehat{\Pi \psi^{-1}}$. Naturally follows immediately, as the diagram of sites is natural in $X$.
\end{proof}
\begin{lemma}\label{SimpBerk}
Let $X$ be a proper variety over $k$. Then $S_\bullet (X^{an})$ is isomorphic to $\Pi X^{an}_t$ and $\Pi X^{an}_{ft}$ in $\mathrm{Pro}-\mathrm{Ho}(s\mathbf{Set})$. 
\end{lemma}
\begin{proof}
Firstly, by Theorem 14.4.1 of \cite{HL}, $X^{an}$  is locally contractible. Moreover, since $X \to \Spec(k)$ is proper so is $X^{an} \to \Spec(k)^{an}$ by Proposition 3.4.7 of \cite{SpectralGeometry}. Therefore $X^{an}$ is compact, and therefore paracompact. We can therefore apply Theorem 12.1 of \cite{AM}, to get $\Pi X^{an}_t \cong S_\bullet(X^{an})$.  This is isomorphic to $\Pi X^{an}_{ft}$ by Lemma $\ref{verdierfunctorsingular}$. 
\end{proof}

\begin{cor}\label{canon}
Let $X$ be a smooth proper variety. Then there is a natural morphism in $\mathrm{Pro}-\mathrm{Ho}(s\mathbf{Set})$:
$$
\Et^\natural(X) \to \widehat{S_\bullet(X^{an})}^\natural.
$$
\end{cor}
\begin{proof}
By Corollary $\ref{SimpBerk}$, we can identify $\Pi(X^{an}_{ft})$ with $S_\bullet(X^{an})$, and since $X$ is smooth, it is unibranch, so $\Et(X) = \widehat{\Et}(X)$ by Theorem 11.1 of \cite{AM}. This gives us a morphism $\Et^{\natural}(X) \to \widehat{S_\bullet(X^{an})}$, given by $\widehat{\Pi \phi}^{\natural} \circ \widehat{\Pi \psi^{-1}}^{\natural}$.
\end{proof}
In the case above, we can remove the assumption that $X$ is smooth, but we have to add a profinite completion to $\Et^{\natural}(X)$. For the purpose of this paper, restricting to the smooth case is sufficient, but it may be possible to remove this assumption. In the sequel, we wish to construct an equivariant version of the previous theorem, i.e., a natural morphism
$$
\Et^{\natural}_{/k}(X) \to \widehat{(X^{an}_{\hat{\kbar}})}^{\natural} \in \mathrm{Pro}-\mathrm{Ho}(s\mathrm{Gal}_k-\mathbf{Set}).
$$

\subsection{An Equivariant Comparison Morphism}

\begin{defn}
Let $L/k$ be a finite Galois extension. Let $Y \to X^{an}$ be a morphism of Berkovich spaces. Then define $Y_L := Y \times_{X^{an}} (X_L)^{an}$, where this fibre product is taken in the category of Berkovich spaces.

Let $\mathcal{C}$ be one of the sites $X_{\et}$, $X^{an}_{\et}, X^{an}_{ft}$, and let $\mathcal{C}_L$ denote $(X_L)_{\et}, (X^{an}_L)_{\et}$ or $(X^{an}_L)_{ft}$ respectively. Let $U_{\bullet}$ be a hypercovering in $\mathcal{C}$. Then we see $(U_{\bullet})_L$ is a hypercovering in $\mathcal{C}_L$, and $\pi_0((U_{\bullet})_L)$ is an object of the category $s\mathrm{Gal}(L/k)-\mathbf{Set}$. Define
$$
\Pi_{L/k} \mathcal{C} := \{ \pi_0((U_{\bullet})_L) \}_{U_\bullet \in HC(\mathcal{C})} \in \mathrm{Pro}-\mathrm{Ho}(s\mathrm{Gal}(L/k)-\mathbf{Set}).
$$

Note that for these choices of $\mathcal{C}$, the assignment $\mathcal{C} \mapsto \Pi_{L/k} \mathcal{C}$ is functorial. That is, the diagram of morphisms of sites
\begin{center}
\begin{tikzcd}
& X^{an}_{\et} \ar[ld] \ar[rd] & \\
X_{\et} && X^{an}_{ft},
\end{tikzcd}
\end{center}
gives rise to a diagram of morphisms in $\mathrm{Pro}-\mathrm{Ho}(s\mathrm{Gal}(L/k)-\mathbf{Set})$
\begin{center}
\begin{tikzcd}
& \Pi_{L/k} X^{an}_{\et} \ar[ld, "\psi_{L/k}"'] \ar[rd ,"\phi_{L/k}"] & \\
\Pi_{L/k} X_{\et} && \Pi_{L/k} X^{an}_{ft}.
\end{tikzcd}
\end{center}
\end{defn}
\begin{lemma}\label{proet}
There is an isomorphism in $\mathrm{Pro}-\mathrm{Ho}(s\mathrm{Gal}_k-\mathbf{Set})$
$$
\Et_k^{\natural}(X) \cong \{ (\Pi_{L/k} X_{\et})^{\natural}\}_{L/k},
$$
where $L/k$ runs over all finite Galois extensions.
\end{lemma}
This lemma is an easy consequence of the proof of Proposition 9.19 of \cite{HS}, and the style of argument is identical. We recall the argument here since we adapt this argument to the Berkovich case later.
\begin{proof}
For $U_\bullet \to X$ a hypercovering, we can think of $\pi_0( (U_\bullet)_L)$ as an element of $s\mathrm{Gal}_k-\mathbf{Set}$ by letting the $\Gal_k$ action factor through $\mathrm{Gal}(L/k)$.  We adopt the notation of \cite{HS} for this proof. If $U_\bullet \to X$ is a hypercovering, define 
$$
\mathbf{X}_{U,n} := \mathrm{Ex}^{\infty}(\mathrm{cosk}_n(\pi_0(U_\bullet)_{\kbar})).
$$ Similarly, for $L/k$ a finite Galois extension, define
$$
\mathbf{X}_{U,L,n} := \mathrm{Ex}^{\infty}(\mathrm{cosk}_n(\pi_0(U_\bullet)_L)).
$$
In this notation, we see $\Et_k^{\natural}(X) = \{ \mathbf{X}_{U,n}\}_{U,n}$, and $\{ (\Pi_{L/k} X_{\et})^{\natural}\}_{L/k} = \{ \mathbf{X}_{U,n,L}\}_{U,n,L}$.

Let $U_\bullet \to X$ be an étale hypercovering, and let $n \geq 0$. For $0 \leq i \leq n$, let $L_i$ be a finite Galois extension of $k$ such that every connected component of $(U_i)_{\kbar}$ is defined over $L_i$, so that the action of $\mathrm{Gal}_k$ on $\pi_0(U_i)$ factors through $\mathrm{Gal}(L_i/k)$. Let $L$ be a finite Galois extension of $k$ that contains every $L_i$. 

Note $\mathbf{X}_{U,n}$ depends only on $\pi_0 ((U_i)_{\kbar})$ for $0 \leq i \leq n$. However, since each $U_i$ splits over $L$, we see that $\mathbf{X}_{U,n} = \mathbf{X}_{U,n,L}$, so for every $U_\bullet, n$, the object $\mathbf{X}_{U,n}$ appears in the diagram defining $\{ (\Pi_{L/k} X_{\et})^{\natural}\}_{L/k}$, so there is a morphism
$$
\{ (\Pi_{L/k} X_{\et})^{\natural}\}_{L/k} \to \Et_k^{\natural}(X).
$$

Note though that the $\mathbf{X}_{U,n}$s are cofinal in this diagram. Let $L$ be the field extension corresponding to $\mathbf{X}_{U,n}$ as above. Then for every $L'/k$ a finite Galois extension containing $L$, we see $\mathbf{X}_{U,n,L'} = \mathbf{X}_{U,n,L} = \mathbf{X}_{U,n}$. Now let $k'$ be any finite Galois extension of $k$, and let $L'$ be a finite Galois extension of $k$ containing both $L$ and $k'$. Then there is a canonical map $\mathbf{X}_{U,n} = \mathbf{X}_{U,n,L'} \to \mathbf{X}_{U,n,k'}$, so the $\mathbf{X}_{U,n}$ are cofinal in the diagram, as required.
\end{proof}

\begin{thm}\label{GaloisEquivariant}
Let $F$ be the forgetful functor 
$$
\mathrm{Pro}-\mathrm{Ho}(s\mathrm{Gal}(L/k)-\mathbf{Set})\to \mathrm{Pro}-\mathrm{Ho}(s\mathbf{Set}).
$$
 Then for our choice of $\mathcal{C}$ as above, there is an isomorphism $F(\Pi_{L/k}\mathcal{C}) \cong \Pi \mathcal{C}_L$.
\end{thm}
In order to prove the above, we adapt the proof of Proposition 9.19 of \cite{HS}. The proof of this statement uses Weil restrictions for the étale hypercoverings that define the étale homotopy type in order to show that étale hypercoverings defined over $C$ are cofinal among all étale hypercoverings of $C_L$. We therefore start by considering Weil restrictions for Berkovich spaces.

\begin{defn}We say that $\mathcal{X}$ is \emph{good} if for each $x \in \mathcal{X}$, there is a neighbourhood of $x$ isomorphic to $\mathcal{M}(A)$ for some affinoid $k$ algebra, $A$. We want to form Weil restrictions of Berkovich spaces in much the same way as for schemes.  Let $S, S'$ be Berkovich spaces, and let $S' \to S$ be a morphism of Berkovich spaces. Let $\mathcal{X}$ be a Berkovich space over $S'$, so we have a morphism $\mathcal{X} \to S'$. Let $\mathcal{R}_{S'/S}(\mathcal{X})$ be the functor from Berkovich spaces over $S$ to sets given by
$$
\mathcal{Y} \mapsto \mathrm{Hom}_{/S'}( \mathcal{Y} \times_S S' , \mathcal{X}).
$$
 If $\mathcal{R}_{S'/S}(\mathcal{X})$ is representable as a Berkovich space, write $R_{S'/S}(\mathcal{X})$ to be the Berkovich space representing it. Theorem 3.3.2 of \cite{WeilRestrictions} says that $R_{S'/S}(\mathcal{X})$ exists if $S' \to S$ is finite and free and if for all finite subsets $I \subseteq \mathcal{X}$, there exists an affinoid domain that is a neighbourhood of $I$.
 
 In particular, $R_{S'/S}(\mathcal{X})$ is representable if $S = X^{an}$ and $S'=X^{an}_L$ for $X$ a variety over $k$, and if $\mathcal{X}$ is a satisfies the required condition. If $\mathcal{X}=S'$, then it is easy to check that $R_{S'/S}(S') = S$. 
\end{defn}

\begin{thm}\label{3.3.2crit}
Let $X/k$ be a variety, $L/k$ a finite Galois extension, and suppose that $\phi\colon \mathcal{Y} \to X_L^{an}$ is a finite morphism of Berkovich spaces. Then $\mathcal{Y}$ is good, and for all finite subsets $I \subseteq \mathcal{Y}$, there exists an affinoid domain of $Y$ that is a neighbourhood of $I$. 
\end{thm}
\begin{proof}
By Proposition 3.17 of \cite{SpectralGeometry}, if $V \subseteq X_L^{an}$ is an affinoid domain then $\phi^{-1}(V)$ is as well, so $\mathcal{Y}$ is good. Similarly, let $V$ be an affinoid domain containing $\phi(I)$, which exists by Proposition 3.5.1 of \cite{WeilRestrictions}. Then $\phi^{-1}(V)$ is an affinoid domain of $Y$ that is a neighbourhood of $I$.
\end{proof}

We require an infinitesimal lifting criteria for étaleness. This is likely well known, though the author was unable to find a reference, so we recall the statement and sketch a proof here for convenience.
\begin{lemma}
Let $\mathcal{M}(A),\mathcal{M}(B)$ be affinoid $k$-analytic spaces where $A,B$ are affinoid $k$-algebras and let $f: \mathcal{M}(B) \to \mathcal{M}(A)$ be a morphism with finite fibres. Then $f$ is étale if and only if for all affinoid $C$ with an admissible morphism $A \to C$ and $J$ an ideal of $C$ with $J^2=0$, there is a bijection
$$
\Hom_{A}(B,C) = \Hom_A(B, C/J).
$$
\end{lemma}
\begin{proof}
The proof of this is identical to the proof of the infinitesimal lifting criterion for étaleness of schemes (see \S2.2, Proposition 6 of \cite{NeronModels}), making the appropriate modifications where we need to consider completed tensor products.
\end{proof}

\begin{lemma}
Let $U$ be an affinoid $k$-analytic space, and let $f: V \to U_L$ be a local isomorphism (resp. étale cover) with finite fibres. Then the relative Weil restriction $R_{U_L/U}(V)$ exists and the canonical map $R_{U_L/U}(V) \to U$ is a local isomorphism (resp. étale cover) with finite fibres.
\end{lemma}
\begin{proof}
The existence comes by applying Theorem $\ref{3.3.2crit}$, which allows us to apply Theorem 3.3.2 of \cite{WeilRestrictions} to show that $R_{U_L/U}(V)$ exists. The proof of Theorem 3.3.2 and 3.3.1 of this paper proceed roughly as follows. Since $V$ is good, for each $v \in V$, there is an affinoid domain $V_v$ containing $v$, $V_v$. We can form $R_{U_L/U}(V_v)$ in a similar way to forming Weil restrictions of rings, by looking at the underlying affinoid algebras, and verifying that the Weil restriction of an affinoid algebra is affinoid. We then take the analytic spectrum, and we glue the $R_{U_L/U}(V_v)$s back together. In particular, the Weil restriction $R_{U_L/U}(V)$ admits an open cover by $R_{U_L/U}(V_v)$, where $V_v$ is a $k$-affinoid space containing $v$.

Suppose that $V \to U_L$ is étale. We claim that $R_{U_L/U}(V) \to U$ is also étale. By an identical argument to the associated argument for schemes (see Proposition A.5.2(4) of \cite{NeronModels}), we see that $R_{U_L/U}(V) \to R_{U_L/U}(U)=U$ is étale. Note that to invoke this argument, we only need characterisation of étale morphisms in terms of infinitesimal lifting criteria, which is the previous lemma. 

Suppose that $V \to U_L$ is a local isomorphism. We claim that $R_{U_L/U}(V)$ is a local isomorphism as well. By the proof of Lemma 3.3.1 and 3.3.2 of \cite{WeilRestrictions}, we see that $R_{U_L/U}(V)$ has an open cover by sets $R_{U_L/U}(V_v)$ for each $v \in V$ where $V_v$ is an open affinoid domain of $V$ containing $v$. By Lemma 3.3.1 of \cite{WeilRestrictions}, since $R_{U_L/U}(V_v)$ exists, if $V'_v$ is an open subspace of $V_v$, then $R_{U_L/U}(V'_v)$ also exists, and moreover by the same argument as for Theorem 3.3.2 of \cite{WeilRestrictions}, the sets $R_{U_L/U}(V'_v)$ cover $R_{U_L/U}(V)$. 

Since $f$ is a local isomorphism, around each point $v$, there exists $\hat{V}_v$ such that $f|_{\hat{V}_v}$ is an isomorphism onto its image. Therefore, taking $V'_v$ to be small enough, we only need to show that $R_{U_L/U}(V'_v) \to U$ is an open immersion for each $v$, so that it is an isomorphism onto its image. This then follows since  these Weil restrictions preserve open immersions as the map $U_L \to U$ is proper. The argument for Weil restriction of schemes is Proposition 7.6(i) of \cite{NeronModels}, though the argument for Berkovich spaces is identical.
\end{proof}

\begin{lemma}
Let $\mathcal{Y} \to X_L^{an}$ be a local isomorphism (resp. étale cover) with finite fibres (resp. étale cover with finite fibres). Then there exists $\hat{\mathcal{Y}} \to X^{an}$ which is a local isomorphism (resp. étale cover) with finite fibres such that $\hat{\mathcal{Y}}_L \to X^{an}_L$ factors through $\mathcal{Y}$.
\end{lemma}
\begin{proof}
Let $U_1, \ldots, U_n$ be an open cover of $X^{an}$ such that each $U_i$ is $k$-affinoid. Then $\coprod U_i \to X^{an}$ is an open cover of $X^{an}$ and so $\coprod (U_i)_L \to X^{an}_L$ is an open cover of $X^{an}_L$.

Write $\mathcal{Y}_i := \mathcal{Y} \times_{X_L^{an}} (U_i)_L$. We then see that $\coprod \mathcal{Y}_i \to \coprod (U_i)_L \to X^{an}_L$ is a local isomorphism (resp. étale cover) with finite fibres as well. By the previous lemma, we can consider the Weil restriction, $R_{(U_i)_L/(U_i)}(\mathcal{Y}_i) \to U_i$, which is a local isomorphism (resp. étale cover). Define $\hat{\mathcal{Y}} := \coprod_i R_{(U_i)_L/(U_i)}(\mathcal{Y}_i)$. Then the composition $\hat{\mathcal{Y}} \to \coprod U_i \to X^{an}$ is a local isomorphism (resp. étale cover) with finite fibres.

Moreover, base changing to $L$, we have $\hat{\mathcal{Y}}_L \to X^{an}_L$ factors through $\coprod_i \mathcal{Y}_i \to X^{an}_L$, since each $R_{(U_i)_L/(U_i)}(\mathcal{Y}_i)_L \to (U_i)_L$ factors through $\mathcal{Y}_i$. Finally, we see $\hat{\mathcal{Y}} \to X^{an}_L$ factors through $\mathcal{Y}$, since $\coprod_i \mathcal{Y}_i \to X^{an}_L$ does.
\end{proof}

\begin{proof}[Proof of Theorem $\ref{GaloisEquivariant}$]
The proof of this is closely related to the proof of Proposition 9.19 of \cite{HS}. We will write the argument for $\mathcal{C} = X^{an}_{\et}$,  but the argument for $X^{an}_{ft}$ is identical.

The diagram defining $F(\Pi_{L/k} X^{an}_{\et})$ is contained in the diagram defining $\Pi (X^{an}_L)_{\et}$: that is, if $U_\bullet \to X^{an}$ is a hypercovering, then $(U_\bullet)_L \to X^{an}_L$ is also a hypercovering, so $\pi_0((U_\bullet)_L)$ is contained in the diagram defining $\Pi X^{an}_{\et}$. This gives us a canonical morphism
$$
F(\Pi_{L/k} X^{an}_{\et}) \to \Pi (X^{an}_L)_{\et}.
$$
We claim that this is an isomorphism. For this to be true, it is enough to show that hypercoverings of the form $(U_\bullet)_L$ are cofinal in this diagram. This is follows since if $V_\bullet \to X^{an}_L$ is a hypercovering, it is dominated by a hypercoverings defined over $k$ by the previous lemma. This proves the result.
\end{proof}
\begin{cor}
The morphism $\widehat{(\Pi_{L/k} X^{an}_{\et})}^{\natural} \to \widehat{\Pi_{L/k} X_{\et}}^{\natural}$ is an isomorphism.
\end{cor}
\begin{proof}
Let $F$ denote the forgetful functor $\mathrm{Pro}-\mathrm{Ho}(s\Gal(L/k)-\mathbf{Set}) \to \mathrm{Pro}-\mathrm{Ho}(s\mathbf{Set})$. The previous theorem implies that there are isomorphisms
\begin{align*}
F(\widehat{(\Pi_{L/k} X^{an}_{\et})}^{\natural} ) &\cong \widehat{\Pi (X^{an}_L)_{\et}}^{\natural} \\
F(\widehat{(\Pi_{L/k} X_{\et})}^{\natural} ) &\cong \widehat{\Pi (X_L)_{\et}}^{\natural}.
\end{align*}
By construction, we see that these isomorphisms fit into a commutative diagram
\begin{center}
\begin{tikzcd}
F(\widehat{(\Pi_{L/k} X^{an}_{\et})}^{\natural} ) \ar[r, "\cong"] \ar[d, "F(\psi_{L/k})"] &\widehat{\Pi (X^{an}_L)_{\et}}^{\natural} \ar[d, "\cong"]\\
F(\widehat{(\Pi_{L/k} X_{\et})}^{\natural} ) \ar[r, "\cong"] & \widehat{\Pi (X_L)_{\et}}^{\natural},
\end{tikzcd}
\end{center}
where the right hand vertical arrow denotes the isomorphism from Theorem $\ref{nonequivariantisomorphism}$. In particular, we se that after applying the forgetful functor $F$, the map $\widehat{(\Pi_{L/k} X^{an}_{\et})}^{\natural} \to \widehat{\Pi_{L/k} X_{\et}}^{\natural}$ is an isomorphism. The result follows since the forgetful functor reflects isomorphisms.
\end{proof}
\begin{cor}\label{finiteequivariant}
There is a morphism in $\mathrm{Pro}-\mathrm{Ho}(s\Gal(L/k)-\mathbf{Set})$:
$$
\Et_{L/k}^{\natural}(X) \to \widehat{S_\bullet(X^{an}_L)}^{\natural}.
$$
\end{cor}
\begin{proof}
We clearly have a morphism $\Et_{L/k}^{\natural}(X) \to \widehat{\Pi_{L/k}X^{an}_{ft}}^{\natural}$. The singular functor $S_\bullet$ gives an equivalence of categories
$$
S_\bullet\colon \mathrm{Ho}(sG-\mathbf{Set}) \cong \mathrm{Ho}(G-\text{CW-complexes}).
$$
By Lemma $\ref{SimpBerk}$, we have that $F( \Pi_{L/k}X^{an}_{ft}) \cong F(S_\bullet (X^{an}_L)) \in \mathrm{Pro}-\mathrm{Ho}(s\mathbf{Set})$, where $F$ is the forgetful functor. Since $F$ reflects isomorphisms, $\Pi_{L/k} X^{an}_{ft} \cong S_\bullet (X^{an}_L)$, as required.
\end{proof}

\begin{cor}\label{hfpmap}
There is a morphism
$$
\Et_{k}^{\natural}(X) \to \widehat{S_\bullet(X^{an}_{\hat{\kbar}})}^{\natural} \in \mathrm{Pro}-\mathrm{Ho}(s\mathrm{Gal}_k-\mathbf{Set}),
$$
so by taking the functors $\pi_0( (-)^{h\mathrm{Gal}_k})$, we get a morphism
$$
X(hk) \to \pi_0((\widehat{S_\bullet (X_{\hat{\kbar}}^{an})}^{\natural})^{h\mathrm{Gal}_k}) =: X^{an}(hk).
$$
\end{cor}
\begin{proof}
Apply Corollary $\ref{finiteequivariant}$ to obtain
$$
\Et_{L/k}^{\natural}(X) \to \widehat{S_\bullet(X^{an}_{L})}^{\natural}.
$$
for all $L$. Combining this with the description of $X^{an}_{\hat{\kbar}}$ and $\Et_k(X)$ as pro-objects from Theorems $\ref{proberk}$ and $\ref{proet}$ respectively obtains the desired morphism. It should be noted that Definition 4.1 of \cite{QuickPf} gives us that $\widehat{ S_\bullet(X_{\hat{\kbar}}^{an})}  = \{ \widehat{S_\bullet (X_L^{an})} \}_{L/k}$, where $L/k$ runs over the finite Galois extensions of $k$.
\end{proof}

\section{Fixed Point Theorems}\label{fixedpointsection}
\subsection{Curves}
In this section, we prove a weak version of the main theorem from \cite{PS}, which we also extend to any non-Archimedean valued field, and curves of genus $1$.
\begin{thm}\label{psfixpt}
Let $X/k$ be a smooth proper geometrically connected curve of genus $\geq 1$. Then there is a natural map
$$
\Psi_X\colon X(hk) \to \varprojlim_L \pi_0((X^{an}_{L})^{\mathrm{Gal}(L/k)}).
$$
\end{thm}
\begin{proof}
Let $x \in X(hk)$. Corollary $\ref{hfpmap}$ gives rise to a map $X(hk) \to X^{an}(hk)$, so by the description of $X^{an}_{\hat{\kbar}}$ as a limit in Lemma $\ref{proberk}$, the image of $x$ in $X^{an}(hk)$ gives us an element
$$
W_L \in \pi_0((\widehat{S_\bullet(X^{an}_{L})}^{\natural})^{h\mathrm{Gal}_k})
$$
for every $L/k$ a finite Galois extension. Let $S_L$ be the CW-complex from Theorem $\ref{Skeleton}$, so that there is a $\mathrm{Gal}(L/k)$ equivariant strong deformation retract from $X^{an}_L$ onto $S_L$. For the rest of this theorem, we will think of $S_L$ as a simplicial set, so that there is an isomorphism in $\mathrm{Pro}-\mathrm{Ho}(s\Gal_k-\mathbf{Set})$: $S_\bullet(X^{an}_L) \cong S_L$.

Since $X$ is a curve, $S_L$ is a graph, and so has vanishing higher homotopy groups, therefore $\widehat{S_L}^{\natural} \cong \widehat{S_L}$.  We now apply Lemma $\ref{decomp}$ to write the homotopy fixed points as
$$
\widehat{S_L}^{h\mathrm{Gal}_k} = (\widehat{S_L}^{h\mathrm{Gal}_L})^{h\mathrm{Gal}(L/k)}.
$$
The action of $\Gal_k$ on $S_L$ factors through $\Gal(L/k)$ by construction. Therefore the $\Gal_L$ action on $S_L$ is trivial, so we can apply Lemma $\ref{profmil}$ to see that $\widehat{S_L}^{h\mathrm{Gal}_L}$ is weakly equivalent to $\widehat{S_L}$. In particular, we obtain an isomorphism
$$
\widehat{S_L}^{h\mathrm{Gal}_k} \cong \widehat{S_L}^{h\mathrm{Gal}(L/k)}.
$$
 Therefore consider $W_L$ as an element of $\pi_0(\widehat{S_L}^{h\mathrm{Gal}(L/k)})$. We then apply the ``section conjecture for graphs", Theorem 3.7 of \cite{HarpGraph}, to obtain a bijection  
$$
\psi\colon \pi_0(\widehat{S_L}^{h\mathrm{Gal}(L/k)}) \to \pi_0(S_L^{\mathrm{Gal}(L/k)}).
$$
This in turn gives us an element $\psi(W_L)$ lying in $\pi_0((X_L^{an})^{\mathrm{Gal}(L/k)})$. Note that if $L'/L$ is a field extension with $L'/k$ Galois, then $\psi(W_{L'})$ must map to $\psi(W_L)$ under the map $(X_{L'}^{an})^{\Gal(L'/k)} \to (X_L^{an})^{\Gal(L/k)}$. Defining
\begin{align*}
\Psi_X\colon X(hk) &\to \varprojlim_L \pi_0((X^{an}_{L})^{\mathrm{Gal}(L/k)}) \\
x &\mapsto \{ \psi(W_L) \}_{L} \in \varprojlim \pi_0((X_L^{an})^{\mathrm{Gal}(L/k)})
\end{align*}
gives the required map. Naturality of $\Psi_X$ follows by naturality of our constructions.
\end{proof}
\begin{lemma}\label{weakps}
Let $x \in X(hk)$, and let $\Phi_X(x) = (W_L) \in \varprojlim_L \pi_0( (X_L^{an})^{\Gal(L/k)})$. Then there exists $\overline{\gamma} \in (X^{an}_{\hat{\kbar}})^{\Gal_k}$ such that if we let $(\gamma_L)$ denote the image of $\overline{\gamma}$ in $(X^{an}_L)^{\Gal(L/k)}$, then $\gamma_L$ lies in $W_L$. 
\end{lemma}
\begin{proof}
Note that $(X_L^{an})^{\mathrm{Gal}(L/k)}$ is a closed subset of $(X_L^{an})$, so is compact, and therefore each connected component is compact as well. Therefore, $\Phi_X(x) = (W_L)$ is an inverse limit of non empty compact sets.  Moreover, we see that $\Phi_X(x) \subseteq \varprojlim_L X_L^{an} \cong X_{\hat{\kbar}}^{an}$, and so $\Phi_X(x)$ is a non-empty, compact, subset of $X^{an}_{\hat{\kbar}}$. Clearly $\Phi_X(x)$ is fixed by the $\Gal_k$ action however, since it is fixed by the $\Gal_k$ action in every quotient, and so $\Phi_X(x)$ is a non empty subset of $(X^{an}_{\hat{\kbar}})^{\Gal_k}$, as required.
\end{proof}
The above allows us to consider $\Phi_X(x)$ as a non-empty subset of $(X^{an}_{\hat{\kbar}})^{\Gal_k}$. Let $K$ denote the function field of $X$, and consider $X^{an}$ as a subset of $\mathrm{Val}_v(K)$ using Lemma $\ref{analytificationvaluation}$. This embedding respects the Galois actions, so we obtain the following corollary.
\begin{cor}
Let $X/k$ be a smooth, proper, geometrically connected curve of genus $\geq 1$ such that $X(hk) \neq \emptyset$. Then there exists $w$, a valuation on $K \otimes_k \hat{\kbar} = k(X_{\hat{\kbar}})$ such that $w$ is fixed under the natural action of $\mathrm{Gal}_k$ on $\mathrm{Val}_v(K \otimes_k \hat{\kbar})$.  This also gives us a valuation on $K \otimes_k \kbar$ fixed by the natural action of $\mathrm{Gal}_k$ by restricting $w$ to $K \otimes_k \kbar$. 
\end{cor}
\begin{rem}
The above corollary only requires that $X$ is $K(\pi,1)$ and $X^{an}_L$ admits a $\Gal(L/k)$-equivariant strong deformation retraction onto a graph. In particular, this holds for curves of genus $1$, as well as curves of genus $\geq 2$ which are the setting of \cite{PS}.
\end{rem}
For the rest of this section, we show how the above implies the existence part of the main theorem of \cite{PS} using a technique of Tamagawa from Proposition 0.7 of \cite{Tamagawa}. In order to do this, we first need some terminology from this paper. 
\begin{defn}\label{defnprocovers}
Let $X/k$ be a smooth proper geometrically connected variety with function field $K$. Let $\tilde{X}$ denote the pro-scheme which is the universal pro-étale cover of $X$, i.e., $\tilde{X} = \{Y\}$ where $Y$ runs over all finite étale covers of $X$. Let $\tilde{K}$ denote the function field of $\tilde{X}$. Let $\overline{\eta}$ be a geometric point of $X$ lying over the generic point of $X$. Suppose  $s: \Gal_k \to \pi_1^{\et}(X, \overline{\eta})$ is a section. We say the \emph{decomposition tower} of $s$ is the pro-scheme $X^s := \{Y\}$, where $Y$ runs over all connected finite étale covers $Y \to X$ of $X$ such that $s$ factors through $\pi_1^{\et}(Y, \overline{\eta_Y})$, where $\overline{\eta}_Y$ is a choice of geometric point lying over $\overline{\eta}$. Write $K^s$ to mean the function field of $X^s$. Let $\eta$ denote the generic point of $X$, and for $Y \to X$ a finite étale cover where $Y$ is connected, write $\eta_Y$ for the generic point of $Y$ and $K_Y$ for the function field of $Y$. Note that if $Y \to X$ is a finite étale cover in the decomposition tower of $s$, then the maps $\tilde{X} \to X^s \to Y \to X$ induce inclusions $X \subseteq K_Y \subseteq K^s \subseteq \tilde{K}$.
\end{defn}
\begin{rem}\label{tamagawarem}
Note that Proposition 0.7 of \cite{Tamagawa} shows that the pro-scheme $X^s$ allows us to categorise whether the section $s$ is a section in the image of the map $X(k) \to \mathcal{S}_{X/k}$. This Proposition shows that $s$ is the image of $x \in X(k)$ if and only if $x$ is in the image of $X^s(k) \to X(k)$. In particular, $X^s(k)=\emptyset$ implies that $s$ does not come from a $k$-point. When $X$ is a $K(\pi,1)$ variety, recall that the map $X(hk) \to \mathcal{S}_{X/k}$ is a bijection, so a homotopy fixed point $x \in X(hk)$ lies in the image of $h_{X/k}$ from Definition $\ref{hfpdef}$ if and only if $X^s(k) \neq \emptyset$ where $s$ is a section coming from $s$. 
\end{rem}

In order to prove the existence part of the main theorem of \cite{PS}, we show that some of the results from Lemma $\ref{weakps}$ extend to $X^s$.

\begin{lemma}\label{sectiongaloisaction}
Let $s\colon \mathrm{Gal}_k \to \pi_1^{\et}(X, \overline{\eta})$ be a section as above. Then $\tilde{K} \cong K^s \otimes_k \kbar$, and under this isomorphism $s(\Gal_k)$ acts on $\tilde{K}$ by the $\Gal_k$ action on $\kbar$. 
\end{lemma}
\begin{proof}
Let $Y \to X$ be an étale cover of $X$ such that $s$ factors through $\pi_1^{\et}(Y, \overline{\eta_Y})$ for $\overline{\eta_Y}$ a geometric generic point lying over $\overline{\eta}$.  We can identify the étale fundamental group $\pi_1^{\et}(Y, \overline{\eta_Y}) \cong \mathrm{Gal}(\tilde{K}/K_Y)$, and similarly, $\pi_1^{\et}(Y_{\kbar}, \overline{\eta_Y}) \cong \mathrm{Gal}(\tilde{K}/K_{Y_{\kbar}})$. As noted in Definition $\ref{defnprocovers}$, we may consider $K_Y \subseteq K^s \subseteq \tilde{K}$. We claim that $K^s \otimes_k \kbar$ is also subfield of $\tilde{K}$. This is because we have the following sequence of equalities, where the union is taken over all $Y$ in the decomposition tower of $s$:
\begin{align*}
K^s \otimes_k \kbar &= \left(\bigcup_{Y} K_Y\right) \otimes_k \kbar = \bigcup_Y \left( K_Y \otimes_k \kbar \right) = \bigcup_Y \left( K_{Y_{\kbar}} \right) = \bigcup_Y \bigcup_{L/k \text{ finite Galois}} K_{Y_L},
\end{align*}
the final of which is clearly a subfield of $\tilde{K}$ since it is a filtered union of subfields of $\tilde{K}$. Since $K^s \otimes_k \kbar$ is a subfield of $\tilde{K}$, it is a fixed field of some $H$, a closed subgroup of $\mathrm{Gal}(\tilde{K}/K^s) = s(\mathrm{Gal}_k)$. Let $s(\sigma) \in \mathrm{Gal}(\tilde{K}/K^s)$. Then $\sigma$ acts on $K^s \otimes_k \kbar$ by
$$
s(\sigma)(f \otimes \alpha) = f \otimes \sigma(\alpha),
$$
and therefore, it is easy to see that there is no non trivial $s(\sigma)$ such that $s(\sigma)$ fixes $K^s \otimes_k \kbar$. Therefore, $H = \{1\}$, and $K^s \otimes_k \kbar = \tilde{K}$, as required. The statement about the $s(\Gal_k)$ action is immediate, as $K^s$ is the fixed field of $s(\Gal_k)$. 
\end{proof}
\begin{cor}\label{maincor}
Let $X$ be a smooth proper geometrically connected curve of genus $\geq 1$, and let $x \in X(hk)$ with associated section $s$. Then there exists
$$
\overline{\gamma} = (\overline{\gamma_Y}) \in \varprojlim_Y (Y^{an}_{\hat{\kbar}})^{\Gal_k},
$$
where $Y \to X$ runs over all the covers in the decomposition tower of $s$.
\end{cor}
\begin{proof}
 Let $Y \to X$ be a finite étale covering such that $s$ factors through $\pi_1^{\et}(Y, \overline{\eta_Y})$, and write $[s_Y]$ to be the element of $Y(hk)$ represented by $s$. Then $Y$ is again a smooth proper geometrically connected curve of genus $\geq 1$, so we apply Lemma $\ref{weakps}$ to obtain $\Psi_Y([s])$, which is a compact subset of $(Y^{an}_{\hat{\kbar}})^{\Gal_k}$. Another compactness argument gives us the result: each $\Psi_Y([s])$ is non empty and compact, so $\varprojlim_Y \Psi_Y([s])$ is non empty, and we pick $\overline{\gamma}$ to be an element of this set.
\end{proof}
\begin{rem}\label{trivialvaluation}
For this paper, we have restricted to the case where $v$ is a non-trivial valuation on $k$. If the valuation on $k$ is trivial, then we may obtain a trivial version of the above theorem. Let $\eta_Y$ denote the generic point of $Y$. Then we have a canonical point on $Y^{an}_{\hat{\kbar}}$: $\mathrm{trv}_Y$, which is the point lying over the generic point of $Y_{\hat{\kbar}}$ such that the equipped valuation is trivial. Since this valuation is trivial the action of $\Gal_k$ on $Y^{an}_{\hat{\kbar}}$ fixes the point $\mathrm{trv}_Y$. Moreover, these points clearly give us a projective limit, and so we obtain a point $\varprojlim_Y \mathrm{trv}_Y \in \varprojlim_Y (Y^{an}_{\hat{\kbar}})^{\Gal_k}$. In particular, when our base field is trivially valued, the existence of fixed points as above is trivial, regardless of whether $Y$ admits a homotopy fixed point.
\end{rem}
\begin{cor}[Existence part of main theorem of \cite{PS}]\label{existencecor}
Let $X$ be a smooth proper curve of genus $\geq 1$, and let $x \in X(hk)$ with associated section $s$. Then there exists $\tilde{w} \in \mathrm{Val}_v(\tilde{K})$ such that $\tilde{w}$ is fixed by the action of $s(\Gal_k)$.
\end{cor}
\begin{proof}
Let $\overline{\gamma} = (\overline{\gamma}_Y)$ be an element of $\varprojlim_Y (Y^{an}_{\hat{\kbar}})^{\Gal_k}$. For each $Y/k$, the Berkovich analytification $(Y^{an}_{\hat{\kbar}})^{\Gal_k}$ is a subset of $\mathrm{Val}_v(K_Y \otimes_k \hat{\kbar})$ by Lemma $\ref{analytificationvaluation}$. This means we view $\overline{\gamma}$ as an element of
$$
\varprojlim_Y \mathrm{Val}_v(K_Y \otimes_k \hat{\kbar})^{\Gal_k} = \mathrm{Val}_v( \bigcup_Y K_Y \otimes_k \hat{\kbar})^{\Gal_k} = \mathrm{Val}_v(K^s \otimes_{k} \hat{\kbar})^{\Gal_k}.
$$
which we may restrict to obtain an element of $\mathrm{Val}_v(K^s \otimes_k \kbar)^{\Gal_k}$. Applying Lemma $\ref{sectiongaloisaction}$ gives us an isomorphism $\mathrm{Val}_v(K^s \otimes_k \kbar)^{\Gal_k} \cong \mathrm{Val}v(\tilde{K})^{s(\Gal_k)}$. Therefore $\overline{\gamma}$ gives us an element of $\mathrm{Val}v(\tilde{K})^{s(\Gal_k)}$, so this set is non empty. 
\end{proof}

\begin{rem}\label{graphsonly}
For Corollary $\ref{maincor}$, we do not need that $k$ is a $p$-adic field, only that $k$ is a field which is complete with respect to a non-Archimedean valuation. For example, the above Corollary will hold when $k$ is a Laurent series field, or if $k$ is the completion of a global function field. 

Also note that the only time we use that $X$ is a curve of genus $\geq 1$ in the above Corollary $\ref{maincor}$ is Theorem $\ref{psfixpt}$, as we make essential use of Harpaz's section conjecture for graphs. If $X$ is any smooth proper and geometrically connected variety, then we would be able to prove the same result if $X$ is $K(\pi,1)$ and we have a result comparing the homotopy fixed points of $X^{an}$ to the fixed points of $X^{an}$.
\end{rem}
While we obtain a fixed point from the work above, it remains to be seen whether the fixed point we obtain has any compatibility with the original homotopy fixed point. 
\begin{defn}\label{compat}
Suppose that $X/k$ is an arbitrary smooth, proper, geometrically connected variety and let $x \in X(hk)$. Then $x$ gives rise to a section of the fundamental short exact sequence, $s$, and $\Gal_k$ acts on $\varprojlim_Y (Y^{an}_{\hat{\kbar}})$, where $Y \to X$ runs over all finite étale covers in the decomposition tower of $s$. Suppose that we have an element
$$
\overline{\gamma} = (\overline{\gamma_Y}) \in (\varprojlim_Y (Y^{an}_{\hat{\kbar}}))^{\Gal_k},
$$
where $\overline{\gamma_Y} \in \Phi_Y([s_Y])$ for all $Y$. For each $\overline{\gamma_Y}$ in $(Y^{an}_{\hat{\kbar}})^{s(\Gal_k)}$, write $\gamma_Y$ to mean the image of $\overline{\gamma_Y}$ under the canonical map $Y^{an}_{\hat{\kbar}} \to Y^{an}$. By construction, the collection of $\gamma_Y$s are compatible with transition maps between the $Y$s, so they glue together to give
$$
\gamma = (\gamma_Y) \in \varprojlim_Y Y^{an}.
$$
Write $k(\gamma) = \bigcup_Y k(\gamma_Y)$, where $k(\gamma_Y)$ is as in Definition $\ref{markedpt}$. Note that, since the $\overline{\gamma_Y}$ are all fixed under the action of $\Gal_k$, we see that $k$ is separably closed in $k(\gamma)$, and so there is a surjection $\Gal_{k(\gamma)} \to \Gal_k$. 

For every $Y$, write $Y^\gamma$ to denote the base change of $Y$ to $k(\gamma)$. Note that, for each $Y$, the construction from Definition $\ref{markedpt}$ gives rise to a $k(\gamma_Y)$ point on $Y \times_k \Spec(k(\gamma_Y))$, which in turn pulls back to a $k(\gamma)$ point on $Y^\gamma$, which we also call $\gamma_Y$, where it is clear.
\end{defn}

Suppose now that $X/k$ is a smooth proper geometrically connected curve of genus $\geq 1$ and let $x \in X(hk)$ be a homotopy fixed point with associated section $s$. We then may use Corollary $\ref{maincor}$ to obtain a Berkovich point $\gamma_Y \in Y^{an}$ for all finite étale covers, $Y \to X$ in the decomposition tower of $s$, which comes from a fixed point of $Y^{an}_{\hat{\kbar}}$. The above remark applied when $Y=X$ uses $\gamma_X$ to construct a point in $X^\gamma(k(\gamma))$.  Applying the map $h_{X^\gamma}\colon X^\gamma(k(\gamma)) \to X^{\gamma}(hk(\gamma))$ gives us a homotopy fixed point of $X^\gamma$. 

However, we may also obtain a homotopy fixed point of $X^\gamma$ independently of Corollary $\ref{maincor}$, simply by pulling back $x \in X(hk)$ to $X^{\gamma}(hk(\gamma))$. It remains to see whether these two constructions are compatible. 
\begin{lemma}\label{compatibility}
Let $X/k$ be a curve of genus $\geq 1$. Let $x \in X(hk)$, and let $\overline{\gamma}$ be a fixed point as in Corollary $\ref{maincor}$. Let $\gamma_X \in X^\gamma(k(\gamma))$ be the point as constructed above from $(\overline{\gamma})$. Denote the pullback map on homotopy fixed points by $\iota_\gamma\colon X(hk) \to X^\gamma(hk(\gamma))$, and consider the canonical map $h_{X^\gamma}\colon X^\gamma(k(\gamma)) \to X^\gamma(hk(\gamma))$. Then 
$$
h_{X^\gamma}(\gamma_X) = \iota_\gamma(x) \in X^\gamma(hk(\gamma)).
$$
\end{lemma}
\begin{proof}
Let $s$ denote the section associated to $x \in X(hk)$. Let $X^s = \{Y\}$ be the pro-scheme corresponding to the decomposition tower of $s$, and let $(X^{\gamma})^{\iota_\gamma(s)}$ be decomposition tower of the associated section that comes from $\iota_\gamma(x)$.

 By functoriality of the homotopy fixed point map, we see that, as pro-$k(\gamma)$ varieties $(X^\gamma)^{\iota_\gamma(s)} = (X^s)_{k(\gamma)} = \{Y^\gamma\}$.

Consider the point $\gamma_Y \in Y^\gamma(k(\gamma))$. By construction, these $\gamma_Y$ points are compatible with transition maps between finite étale covers, and so we obtain an element
$$
\varprojlim_Y \gamma_Y \in (X^\gamma)^{\iota_\gamma(s)}(k(\gamma))
$$
lying over $\gamma_X$ under the map $(X^\gamma)^{\iota_\gamma(s)}(k(\gamma)) \to X^{\gamma}(k(\gamma))$. As in Remark $\ref{tamagawarem}$, Proposition 0.7 of \cite{Tamagawa} implies that $\iota_\gamma(x) \in X^\gamma(hk(\gamma))$ satisfies $\iota_\gamma(x) = h_{X^\gamma}(\gamma_X)$, as required. 
\end{proof}
\begin{defn}\label{defncompat}
Let $X/k$ be a variety, let $x \in X(hk)$ with associated section $s$. Suppose that for all finite étale covers $Y \to X$ in the decomposition tower, there exists 
$$
\overline{\gamma} = (\overline{\gamma_Y}) \in (\varprojlim_Y (Y^{an}_{\hat{\kbar}}))^{\Gal_k}.
$$
Let
$$
\gamma = (\gamma_Y) \in \varprojlim_Y (Y^{an})
$$
denote the image of $\overline{\gamma}$ as in Definition $\ref{compat}$, and consider $\gamma \in X^\gamma(k(\gamma))$. We say that $\overline{\gamma}$ is \emph{compatible} with the homotopy fixed point $x \in X(hk)$, if $\iota_\gamma(x) = h_{X^\gamma}(\gamma_X) \in X^\gamma(hk(\gamma))$. 
\end{defn}
\begin{rem}
Suppose $x \in X(hk)$ in the above Lemma is such that $x = h_{X}(*)$ where $* \in X(k)$, and let $s$ be an associated section. For each $Y \to X$ in the decomposition tower of $s$, there exists a lift $*_Y$ of $*$ by Proposition 0.7 of \cite{Tamagawa}. Let $\gamma_Y$ be the unique point of $Y^{an}$ lying over the support of $*_Y$, and let $\overline{\gamma_Y}$ be the unique point of $Y^{an}_{\hat{\kbar}}$ lying over $\gamma_Y$. Then $\overline{\gamma_Y}$ is fixed by $\Gal_k$ and $k(\gamma_Y) = k$ for all $Y$. In particular, $\iota_\gamma$ is the identity, $X^{\gamma}=X$ and $\gamma_X = * \in X(k)$. This implies that if we take a homotopy fixed point coming from a $k$-point, it will be compatible with the valuation induced by that $k$-point.
\end{rem}

The above notion makes sense for any variety $X/k$ where $k$ is any field which is complete with respect to a non-Archimedean valuation of rank $1$. It is therefore reasonable to ask the following question.

\begin{qst}
For a fixed field $k$ which classes of varieties $X/k$ are such that any homotopy fixed point $x \in X(hk)$ with associated section $s$ has a compatible system
$$
\overline{\gamma} = (\overline{\gamma_Y}) \in (\varprojlim_Y (Y^{an}_{\hat{\kbar}}))^{\Gal_k},
$$
where $Y$ runs over all finite étale covers in the decomposition tower of $s$?
\end{qst}
Lemma $\ref{weakps}$ guarantees that for any field as above, curves of genus $\geq 1$ satisfy this condition over any field which is complete with respect to a non-archimedean valuation. By restricting the field $k$, we may seek to apply other fixed point theorems that may only apply for fixed points under certain group actions. For example, Remark $\ref{trivialvaluation}$ guarantees that we can find a trivial fixed point as above for any variety and any $k$. In the following subsection, we describe one class of varieties that satisfies the condition of the above question for any field which is complete with respect to a non-archimedean valuation.

\subsection{Varieties Fibred into Proper Hyperbolic Curves}
Theorem 7.1 of \cite{StixSchmidt} suggests that a certain class of varieties, known as strongly hyperbolic Artin neighbourhoods, are in some sense a good candidate for theorems from anabelian geometry when we replace the étale fundamental group with the étale homotopy type. In this subsection, we show that a class of varieties related to strongly hyperbolic Artin neighbourhoods satisfy a version of Lemma $\ref{weakps}$. A key ingredient is a result of Corwin and Schlank, Lemma 9.4 of \cite{CorwinSchlank}, which allows us to work with homotopy fixed points along fibrations. 
\begin{defn}
Let $X/k$ be a geometrically connected variety. We say $X$ is \emph{fibred into proper hyperbolic curves} if $X$ factors as 
$$
X = C_n \to C_{n-1} \to \ldots \to C_1 \to C_0 = \Spec(k)
$$
 such that for any field $L$, and any $y \in C_i(L)$, we have that $(C_{i+1})_y \to \Spec(L)$ is a smooth proper geometrically connected curve of genus $\geq 1$. Note that  $\mathrm{dim}(X) = n$.
\end{defn}
\begin{rem}
A strongly hyperbolic Artin neighbourhood in the sense of Definition 6.1 of \cite{StixSchmidt}, is a variety $X$ admitting a sequence of morphisms
$$
X = C_n \to C_{n-1} \to \ldots \to C_1 \to C_0 = \Spec(k),
$$
where each $C_i \to C_{i-1}$ is an elementary fibration of hyperbolic curves, and each $C_i$ embeds into a product of hyperbolic curves. A proper strongly hyperbolic Artin neighbourhood is therefore fibred into hyperbolic curves. This properness condition is very restrictive. For example, Lemma 6.3 of \cite{StixSchmidt} does not hold if we replace strongly hyperbolic Artin neighbourhoods with varieties fibred into proper hyperbolic curves. While there is an analogue of the section conjecture for non-proper curves, the techniques used in this paper are unable to handle non-proper curves, due to the strong reliance on compactness arguments.
\end{rem}

\begin{thm}
Let $X/k$ be fibred into hyperbolic curves, and suppose $x \in X(hk)$, and let $s$ be a section corresponding to $x$. Then there exists
$$
\overline{\gamma} = (\overline{\gamma_Y}) \in \varprojlim_Y (Y^{an}_{\hat{\kbar}})^{\Gal_k}
$$
where $Y \to X$ runs over all finite étale covers in the decomposition tower of $s$. Moreover, this $\overline{\gamma}$ is compatible with $x$ in the sense of Definition $\ref{defncompat}$.
\end{thm}
\begin{proof}
Since $X$ is fibred into hyperbolic curves, write
$$
X = C_n \to C_{n-1} \to \ldots \to C_1 \to C_0 = \Spec(k).
$$
 We proceed by induction on $n$, with the case $n=1$ being Lemma $\ref{weakps}$. Suppose it is true for all $m < n$, and let $Z:= C_{n-1}$. Since there is a morphism $f\colon X \to Z$, we see that $f(x) \in Z(hk)$, so by induction, there exists $\overline{z} \in \varprojlim_{Y} (Y_{\hat{\kbar}}^{an})^{\Gal_k}$, where $Y$ runs over all finite étale covers of $Z$ in the decomposition tower of the section associated to $f(x)$.

Let $k(z)$ be the field constructed from this system of fixed points as in Lemma $\ref{compatibility}$. Let $\hat{X}, \hat{Z}$ denote the base change of $X,Z$ to $k(z)$. As in Definition $\ref{markedpt}$, we obtain a point, $z_Z \in \hat{Z}(k(z))$, and by the induction hypothesis, $h_{\hat{Z}}(z_Z) = \iota_{z}(f(x)) \in \hat{Z}(hk(z))$.

Let $\hat{X}_z$ denote the fibre of $\hat{X}$ over $z_Z$, so that we have the following commutative pullback diagram of varieties over $k(z)$. 
\begin{center}
\begin{tikzcd}
\hat{X}_z \ar[r] \ar[d] & \hat{X} \ar[d, "\hat{f}"] \\
\Spec(k(z)) \ar[r, "z_Z"] & \hat{Z}
\end{tikzcd}
\end{center}
 Then we see $\hat{X}_z$ is a smooth proper curve of genus $\geq 1$ over $k(z)$. Moreover, by Lemma 9.4 of \cite{CorwinSchlank}, there exists $\hat{x} \in \hat{X}_z(hk(z))$ such that $f(\hat{x}) = h_{\hat{Z}}(z_Z)$. We therefore apply Corollary $\ref{maincor}$ to obtain $\overline{\gamma'} = (\overline{\gamma'_Y}) \in \varprojlim_Y (Y_{\hat{\overline{k(z)}}}^{an})^{\Gal_{k(z)}}$, where $Y$ runs over all finite étale coverings of $\hat{X}$ in the decomposition tower of $\hat{x}$. This allows us to form our inverse limit $(\gamma') = (\gamma'_Y) \in \varprojlim_Y Y^{an}$ as in Definition $\ref{compat}$.

For any $Y \to X$ in the decomposition tower of $x$, we see that $Y_{k(z)} \to \hat{X}$ is in the decomposition tower of $\hat{x}$. Let $\overline{\gamma_Y}$ be the image of $\overline{\gamma'_Y}$ under the canonical map $Y_{\hat{\overline{k(z)}}}^{an} \to Y_{\hat{\kbar}}^{an}$. This map can be seen to be $\Gal_{k(\gamma')}$-equivariant, where $\Gal_{k(\gamma)}$ acts on $Y_{\hat{\kbar}}^{an}$ through the surjection $\Gal_{k(\gamma')} \to \Gal_k$. Therefore each $\overline{\gamma_Y} \in (Y^{an}_{\hat{\kbar}})^{\Gal_k}$, and the $\overline{\gamma_Y}$ glue to a compatible series of points as required.
\end{proof}
\begin{rem}
As noted in Definition 6.1 of \cite{StixSchmidt}, varieties fibred into proper hyperbolic curves are $K(\pi,1)$ varieties. It remains to be seen whether there are varieties which are not $K(\pi,1)$, but nevertheless will satisfy the theorem above.
\end{rem}

\bibliography{EtBerk}

\begin{thebibliography}{Wah09}

\bibitem[AM69]{AM}
M.~Artin and B.~Mazur.
\newblock {\em Etale homotopy}, volume No. 100 of {\em Lecture Notes in
  Mathematics}.
\newblock Springer-Verlag, Berlin-New York, 1969.

\bibitem[BD10]{HFP}
Mark Behrens and Daniel~G. Davis.
\newblock The homotopy fixed point spectra of profinite {G}alois extensions.
\newblock {\em Trans. Amer. Math. Soc.}, 362(9):4983--5042, 2010.

\bibitem[Ber90]{SpectralGeometry}
Vladimir~G Berkovich.
\newblock {\em {Spectral theory and analytical geometry over non-Archimedean
  fields / Vladimir G. Berkovich.}}
\newblock Mathematical surveys and monographs ; no. 33. American Mathematical
  Society, Providence, R.I., 1990.

\bibitem[Ber93]{Berk}
Vladimir~G. Berkovich.
\newblock {Étale cohomology for non-Archimedean analytic spaces}.
\newblock {\em Publications Mathématiques de l'IHÉS}, 78:5--161, 1993.

\bibitem[Ber18]{NAEtHop}
Joseph Berner.
\newblock {Shape Theory in Homotopy Theory and Algebraic Geometry}.
\newblock 6 2018.

\bibitem[BLR90]{NeronModels}
Siegfried Bosch, Werner L\"utkebohmert, and Michel Raynaud.
\newblock {\em N\'eron models}, volume~21 of {\em Ergebnisse der Mathematik und
  ihrer Grenzgebiete (3) [Results in Mathematics and Related Areas (3)]}.
\newblock Springer-Verlag, Berlin, 1990.

\bibitem[Bou89]{Bourbaki}
Nicolas Bourbaki.
\newblock {\em Commutative algebra. {C}hapters 1--7}.
\newblock Elements of Mathematics (Berlin). Springer-Verlag, Berlin, 1989.
\newblock Translated from the French, Reprint of the 1972 edition.

\bibitem[CS20]{CorwinSchlank}
David Corwin and Tomer Schlank.
\newblock {Brauer and Etale Homotopy Obstructions to Rational Points on Open
  Covers}, 2020.

\bibitem[FM86]{FM}
E.M. Friedlander and G.~Mislin.
\newblock {Locally finite approximations of Lie groups I}.
\newblock {\em Invent. Math.}, 83:425--436, 1986.

\bibitem[GR06]{SGA1}
A.~Grothendieck and M.~Raynaud.
\newblock {\em Rev{\^e}tements {\'E}tales et Groupe Fondamental: S{\'e}minaire
  de G{\'e}om{\'e}trie Alg{\'e}brique du Bois Marie 1960/61 (SGA 1)}.
\newblock Lecture Notes in Mathematics. Springer Berlin Heidelberg, 2006.

\bibitem[Gro97]{lettertofaltings}
Alexandre Grothendieck.
\newblock {\em Brief an G. Faltings}, volume~1 of {\em London Mathematical
  Society Lecture Note Series}, page 49–58.
\newblock Cambridge University Press, 1997.

\bibitem[Har13]{HarpGraph}
Yonatan Harpaz.
\newblock {The section conjecture for graphs and conical curves}.
\newblock 2013.

\bibitem[HL16]{HL}
Ehud Hrushovski and Fran\c~cois Loeser.
\newblock {\em Non-archimedean tame topology and stably dominated types},
  volume 192 of {\em Annals of Mathematics Studies}.
\newblock Princeton University Press, Princeton, NJ, 2016.

\bibitem[HS13]{HS}
Y.~Harpaz and T.~M. Schlank.
\newblock {\em Homotopy obstructions to rational points}, page 280–413.
\newblock London Mathematical Society Lecture Note Series. Cambridge University
  Press, 2013.

\bibitem[Jon95]{dejong}
A.J.~De Jong.
\newblock {{\'E}tale fundamental groups of non-Archimedean analytic spaces}.
\newblock {\em Compositio Mathematica}, 97:89--118, 1995.

\bibitem[Mil84]{Miller}
Haynes Miller.
\newblock {The Sullivan Conjecture on Maps from Classifying Spaces}.
\newblock {\em Annals of Mathematics}, 120(1):39--87, 1984.

\bibitem[Moc03]{MochRSC}
Shinichi Mochizuki.
\newblock Topics surrounding the anabelian geometry of hyperbolic curves.
\newblock {\em Galois Groups and Fundamental Groups}, (41):119--165, 2003.

\bibitem[Nic16]{NicaiseBerkSkel}
Johannes Nicaise.
\newblock Berkovich skeleta and birational geometry.
\newblock In {\em Nonarchimedean and tropical geometry}, Simons Symp., pages
  173--194. Springer, [Cham], 2016.

\bibitem[P{\'a}l11]{PalRSC}
Ambrus P{\'a}l.
\newblock The real section conjecture and smith's fixed‐point theorem for
  pro‐spaces.
\newblock {\em Journal of the London Mathematical Society}, 83:353--367, 2011.

\bibitem[PS17]{PS}
Florian Pop and Jakob Stix.
\newblock Arithmetic in the fundamental group of a p-adic curve. on the p-adic
  section conjecture for curves.
\newblock {\em Journal für die reine und angewandte Mathematik}, 2017(725):1
  -- 40, 01 Apr. 2017.

\bibitem[Pá]{AmbrusRealII}
Ambrus Pál.
\newblock Simplicial homotopy theory of algebraic varieties over real closed
  fields, part 2 (in preparation).

\bibitem[Pá15]{AmbrusEt}
Ambrus Pál.
\newblock Étale homotopy equivalence of rational points on algebraic
  varieties.
\newblock {\em Algebra \& Number Theory}, 9:815--873, 05 2015.

\bibitem[Pá22]{AmbrusRealI}
Ambrus Pál.
\newblock Simplicial homotopy theory of algebraic varieties over real closed
  fields, part 1, 2022.

\bibitem[Qui08]{PfHomotopyTheory}
Gereon Quick.
\newblock {Profinite homotopy theory.}
\newblock {\em Documenta Mathematica}, 13:585--612, 2008.

\bibitem[Qui12]{QuickPf}
Gereon Quick.
\newblock {Some Remarks on Profinite Completion of Spaces}.
\newblock {\em Advanced Studies in Pure Mathematics}, 63:413--448, 2012.

\bibitem[Qui15]{QuickSection}
Gereon Quick.
\newblock E{xistence of rational points as a homotopy limit problem}.
\newblock {\em Journal of Pure and Applied Algebra}, 219(8):3466--3481, 2015.

\bibitem[SS16]{StixSchmidt}
Alexander Schmidt and Jakob Stix.
\newblock {Anabelian geometry with étale homotopy types}.
\newblock {\em Annals of Mathematics}, 184(3):817--868, 2016.

\bibitem[Sti13]{StixBook}
Jakob Stix.
\newblock {\em Rational points and arithmetic of fundamental groups}, volume
  2054 of {\em Lecture Notes in Mathematics}.
\newblock Springer, Heidelberg, 2013.
\newblock Evidence for the section conjecture.

\bibitem[Tam97]{Tamagawa}
Akio Tamagawa.
\newblock {The Grothendieck conjecture for affine curves}.
\newblock {\em Compositio Mathematica}, 109(2):135–194, 1997.

\bibitem[Wah09]{WeilRestrictions}
Christian Wahle.
\newblock Weil restriction of p-adic analytic spaces, 2009.

\end{thebibliography}
\bibliographystyle{alpha}

\end{document}